\documentclass[12pt]{amsart}
\usepackage{etex}
\usepackage{amssymb,amsthm,amsfonts,graphicx,color,comment,enumerate,amsmath}
\usepackage{pictex,dcpic}
\usepackage{pdfpages}
\usepackage{tikz}
\usetikzlibrary{patterns,decorations.markings}
\usepackage{verbatim}

\makeatletter
\newtheorem*{rep@theorem}{\rep@title}
\newcommand{\newreptheorem}[2]{%
\newenvironment{rep#1}[1]{%
 \def\rep@title{#2 \ref{##1}}%
 \begin{rep@theorem}}%
 {\end{rep@theorem}}}
\makeatother

\newtheorem{thm}{Theorem}[section]
\newreptheorem{theorem}{Theorem}
\newtheorem{lemma}[thm]{Lemma}
\newreptheorem{lemma}{Lemma}
\newtheorem{cor}[thm]{Corollary}
\newreptheorem{cor}{Corollary}
\newtheorem{prop}[thm]{Proposition}
\newreptheorem{prop}{Proposition}

\theoremstyle{definition}
\newtheorem{definition}[thm]{Definition}

\newtheorem{remark}[thm]{Remark}
\newtheorem{example}[thm]{Example}

\newtheorem{notation}[thm]{Notation}

\def\R{{\mathbb R}}
\def\Z{{\mathbb Z}}

\def\F{{\mathbb F}}

\def\B{{\mathcal B}}
\def\BB{{\mathbf B}}
\def\CC{{\mathbf C}}
\def\JJ{{\mathbf J}}

\def\b{\beta}
\def\i{\iota}

\def\G{{\Gamma}}
\def\e{{\epsilon}}

\begin{document}

\title{Relative Rips Machine}

\author{Pei Wang}
\date{\today}

\begin{abstract}
We study isometric actions of finitely presented groups on $\R$-trees.
In this paper, we develop a relative version of the Rips machine to study \textit{pairs} of such actions. An important example of a \textit{pair} is a group action on an $\R$-tree and a subgroup action on its minimal invariant subtree.
\end{abstract} 

\maketitle
\tableofcontents


\section{Introduction}
The Rips machine is an important tool in the study of group actions on $\R$-trees, see \cite{bfstableactions} \cite{franchrips}. 
In Bestvina and Feighn's version, they describe this machine as moves on a class of $2$-complexes equipped with measured laminations, called \textit{band complexes}.

A \textit{band complex} is constructed in the following way (see more details in Definition~ \ref{bandcomplex}). A \textit{band} $B$ is a space of the form $b\times I$ where $b=[0,a]$ is an arc of the real line and $I=[0,1]$. $b\times \{0\}$ and $b\times\{1\}$ are called \textit{bases} of $B$. 
A \textit{union of bands} $Y$ is a space obtained from a disjoint union $\Gamma$ of intervals by attaching a finite collection of bands through length-preserving homeomorphisms from their bases into $\Gamma$. Each band is assigned a \textit{weight} according to the attaching map. A band $B$ has weight $0$ band if it is an annulus (Definition~\ref{weight}). $\overline{Y}$ denotes the union of bands obtained from $Y$ by omitting weight $0$ bands.
A subset of a band of the form $\{point\}\times I$ is called a \textit{vertical fiber}.
A \textit{leaf} of $Y$ is a class of the equivalence relation generated by $x\sim x'$ if $\{x,x'\}$ is a subset of some vertical fiber of a band. $Y$ is then naturally \textit{foliated} by its leaves. 
A band complex $X$ is a relative $CW$ $2$-complex based on a union of bands $Y$ with 0-, 1-, and 2-cells attached along leaves.


Band complexes arise in the study of group actions on real trees via a process called \textit{resolution} (Definition~\ref{resolution}). Let $G$ be a finitely presented group acting on an $\R$-tree $T$ by isometries, a $G$-tree for short, and $X$ be a band complex with fundamental group $G$. A \textit{resolution} is a $G$-equivariant map $f$ from the universal cover $\widetilde{X}$ to $T$, such that $f$ maps each lift of a leaf to a point in $T$. Every $G$-tree $T$ has a resolution \cite{ms88}, and one may study $T$ through analyzing its resolutions.

The Rips machine is used to classify band complexes. The machine takes as input an arbitrary band complex $X$ and puts it into a normal form -- whose underlying union of bands $Y$ is a finite disjoint union of \textit{components} each of which is one of the following four types: \textit{simplicial}, \textit{surface, toral} and \textit{thin}. The latter three are called \textit{minimal} components. Almost every leaf of a minimal component is dense. 


\vspace{1pc}

In this paper, we develop a relative version of the Rips machine to study \textit{pairs} of group actions on $\R$-trees. To motivate the construction, for a given $G$-tree $T_G$ and a finitely presented subgroup $H<G$, consider a minimal invariant $H$-subtree $T_H\subset T_G$. Let $X_H$ be a band complex resolving $T_H$ and $X_G$ containing $X_H$ be a finite $2$-complex with fundamental group $G$. One may put a band complex structure on $X_G$ by extending the resolution $r: \widetilde{X}_H\to T_H$ to a resolution $\hat{r}: \widetilde{X}_G\to T_G$. Here the two band complexes $X_H$ and $X_G$ along with the inclusion between them, denoted by $(X_H\hookrightarrow X_G)$, is an important example of a \textit{pair of band complexes}. In general, a pair of band complexes $(X\overset{\i}{\to}X')$ consists two band complexes $X$, $X'$ and a \textit{morphism} $\i: X\to X'$ which is locally injective when restricted to the underlying union of bands (Definition~\ref{pair}).

The input of the relative Rips machine is a \textit{pair of band complexes} $(X\overset{\i}{\to}X')$. 
The machine is made of \textit{moves} (Section \ref{background}) on pairs of band complexes. If $(X\overset{\i}{\to}X')$ resolves a pair of trees  $(T\hookrightarrow  T')$ and $(X^*\overset{\i^*}{\to}{X'^*})$ is related to $(X\overset{\i}{\to}X')$ by a sequence of moves, then $(X^*\overset{\i^*}{\to}{X'^*})$ also resolves $(T\hookrightarrow  T')$. This machine converts an arbitrary pair of band complexes into a pair of band complexes in normal forms. Moreover, for a pair $(X\overset{\i}{\to}X')$, eventually only weight $0$ bands in $X$ map into weight $0$ bands in $X'$. Based on these, we have the following results.

Given a pair of band complexes $(X\overset{\i}{\to}X')$ with the underlying pair of unions of bands 
$(Y\overset{\i}{\to}Y')$, let $Y_0\subset Y$ be a component and $Y'_0 \subset Y'$ be the component containing $\i(Y_0)$. 
If $Y'_0$ is simplicial, then so is $Y_0$ since only for simplicial component, every leaf is compact. If $Y'_0$ is a minimal component, we show:
\begin{reptheorem}{t:main} (Part I)
Let $(Y_0 \overset{\i} {\to} Y'_0)$ be a pair of components. Suppose that $Y'_0$ is a minimal component, then $Y_0$ is either a minimal component of the same type as $Y'_0$ or $Y_0$ is simplicial. 
\end{reptheorem}

Using the normal form obtained from the Rips machine, one may check that a finite cover of a minimal component is a minimal component of the same type. The reverse direction is clearly false for toral components (see Remark~\ref{e:infinitetoralcover}). In the case that $Y'_0$ is a surface or thin component, we get:

\begin{reptheorem}{t:main} (Part II)
 $Y_0$ is a surface or thin component if and only if, the relative Rips machine eventually converts $(Y_0 \overset{\i} {\to} Y'_0)$ into a pair $(Y_0^* \overset{\i^*} {\to} Y_0^{'*})$ with the property that $\overline{Y_0^*}$ is a finite cover of $\overline{Y_0^{'*}}$. 
\end{reptheorem}





Now back to our motivation, let $H<G$ be two finitely presented groups and $(X_H\overset{\i}{\to} X_G)$ be a resolving pair for $(T_H\hookrightarrow T_G)$. 
Denote the underlying unions of bands for $X_H$ and $X_G$ by $Y_H$ and $Y_G$ respectively, we have:

\begin{repcor}{output2}
Suppose that $Y_H$ and $Y_G$ are both single surface components or both single thin components and that $\pi_1(Y_H)$ generates $H$, $\pi_1(Y_G)$ generates $G$. If $T_G$ has trivial edge stabilizers, then $[G:H]$ is finite. 

\end{repcor}

Corollary~\ref{output2} paritially generalized Reynold's main result in \cite{rey} to the setting of all finitely presented groups.

\vspace{1pc}



The motivation for building the relative Rips Machine is its applications to the elementary theory of free groups. In section~\ref{s:app}, we define a partial order on band complexes as follows. 

Let $\i:X\to X'$ be a morphism between two band complexes. 
To simplify the notation, suppose that every component of $Y$, the underlying union of bands of $X'$, has a non-empty preimage. Then we say that $X<X'$ if one of the following holds. 
\begin{itemize}
 \item Different components of $Y$ map into the same component of $Y'$. 
 \item A component of $Y$ and its image in $Y'$ have different types.
 \item There exists a component $Y_0\subset Y$ of surface or thin type has the property that the relative Rips machine will eventually convert $\overline{Y}_0$ into a proper cover of its image. 
\end{itemize}

Based on Theorem~\ref{t:main}, we obtain the following proposition which will be used for induction arguments in \cite{definable2}.

\begin{repprop}{FSC}
 Every sequence of band complexes, $X_1<X_2<\dots$, eventually stabilizes. 
\end{repprop}



\vspace{1pc}
The paper is organized as follows. Section~\ref{background} reviews the Rips machine. Among the three types of minimal components, thin type is the most interesting and the least familiar. In Appendix~\ref{thin}, we provide proofs for a few known facts on thin components. As an application, we prove a generalization of Sela's shortening argument for thin components. Section~\ref{RRM} gives basic set-up for the relative Rips Machine. Section~\ref{process} describes the three processes of the relative Rips machine. In Section~\ref{machineoutput}, we discuss the machine output and prove Theorem~\ref{t:main} and Corollary~\ref{output2}. In Section~\ref{s:app}, we define the partial order $<$ on band complexes and prove Proposition~\ref{FSC}.  

\vspace{1pc}
It is a great pleasure to thank my advisor Mark Feighn who introduced me to Rips machine, shared his knowledge and ideas with me, and suggested the possibility of a machine studying pairs of group actions.

\section{Background}\label{background}
In this section we briefly review basic definitions and list some useful properties without proofs about the Rips machine, see detail in \cite{bfstableactions}. 

\begin{definition}
A \textbf{real tree}, or an \textbf{$\R$-tree}, is a metric space $T$ such that for any $x, y$ in $T$ there is a unique arc from $x$ to $y$ and this arc is isometric to an interval of the real line. 

Let $G$ be a group. A $G$-\textbf{tree} T is an \textbf{$R$}-tree with $G$ acting on the left by isometries . A $G$-tree is \textbf{minimal} if it has no proper $G$-subtree.
\end{definition}

Given an \textbf{$\R$}-tree $T$ and $x\in T$. A connected component of $T-x$ is called a \textit{direction} at $x$. We say that $x\in T$ is an \textit{ordinary point} if there are exactly two directions at $x$. A \textit{vertex} is a point which is not an ordinary point. An $\R$-tree is \textit{simplicial} if the set of vertices is discrete.

\begin{definition}
A \textbf{real graph} $\Gamma$ is a finite union of compact simplicial \textbf{$R$}-trees. In particular, each interval of these trees is identified with an arc of the real line.
\end{definition}

Let $I=[0,1]$. A \textit{band} $B$ is a space of the form $b\times I$ where $b=[0,a]$ is an arc of the real line, see Figure~\ref{aband}.  $b\times \{0\}$ and $b\times \{1\}$ are called the \textit{bases} of $B$. A \textit{vertical fiber} is a set of the form $\{point\}\times I$. A subset a vertical fiber is \textit{vertical}, for example the set consisting $\{x, x'\}$ in Figure~\ref{aband}. Similarly, a subset of $b\times\{point\}$ is \textit{horizontal}.

A band has an involution $\delta_{B}$ given by reflection in $b\times\{\frac{1}{2}\}$, known as the \textit{dual map}. 
We may identify $b$ with one base of $B$. Two bases of $B$ then denoted by $b$ and $dual(b)$.

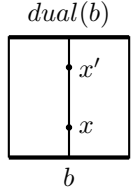
\begin{figure}[h]
   \centering
      \begin{tikzpicture}[thick, scale=0.8]
\footnotesize
\draw (0,0)--(2,0)--(2,2)--(0,2)--(0,0);
\draw (1,0)--(1,2);
\draw[ultra thick](0,0)--(2,0);
\draw[ultra thick](0,2)--(2,2);
\node [above] at (1,2) {$dual(b)$};
\node[below] at (1,0) {$b$};
\coordinate (e) at (1,0.5);
\filldraw(e) circle (1pt) node[right] {$x$};

\coordinate (e) at (1,1.5);
\filldraw(e) circle (1pt) node[right] {$x'$};

\end{tikzpicture}
\caption{\footnotesize The band $B=b\times I$} \label{aband}

\end{figure}

\begin{definition}\label{union of bands}
Let $B_1, \dots, B_n$ be bands and $\Gamma$ be a real graph. For each base $b_i$ of a $B_i$, let $f_{b_i}$ be a length-preserving homeomorphism from $b_i$ to $\Gamma$. A \textbf{union of bands} is the quotient space $Y$ of the disjoint union $\Gamma\sqcup B_1 \sqcup \dots \sqcup B_n$ modulo the union of the $f_{b_i}$'s. 
\end{definition}

\begin{definition}\label{weight}
A band $B=b\times I=[0,a]\times [0,1]$ (or each of its bases) of a union of bands $Y$ is assigned a \textbf{weight} according to the attaching map. $B$ is a weight $0$  band if it is an annulus obtained by attaching $(z,0)$ and $(z,1)$ to the same point of $\Gamma$ for every $z\in b$. $B$ is a weight $\frac{1}{2}$ band if it is a Mobius band obtained by attaching $(z,0)$ and $(a-z,1)$ to the same point of $\Gamma$ for every $z\in b$. Otherwise $B$ is a weight $1$ band. A \textbf{block} is the closure of a connected component of the union of the interiors of the bases, see Figure~\ref{blocks}. The \textbf{complexity of a block} is $\text{max} \{0, -2+\sum {\text{weight}(b)|b\subset \text{block}}\}$. The \textbf{complexity} of $Y$ is the sum of the complexities of its blocks. 
\end{definition}

\begin{figure}[h!]
\begin{tikzpicture}
\footnotesize
\draw[thick] (0,0)--(5,0);
\draw[red][thick](0,-0.05)--(2,-0.05);
\node [below] at (1, -0.05) {$b_1$};
\draw[red][thick](3,-0.05)--(5,-0.05);
\node [below] at (4, -0.05) {$b_3$};
\draw[blue][thick](1,0.05)--(3,0.05);
\node [above] at (2, 0.05) {$b_2$};
\fill (0,0)  circle[radius=2pt];
\fill (3,0)  circle[radius=2pt];
\fill (5,0)  circle[radius=2pt];
\node [left] at (0, 0) {$p_1$};
\node [right] at (5, 0) {$p_3$};
\node [above] at (3, 0) {$p_2$};
\end{tikzpicture}
\caption{\footnotesize Suppose that $b_1, b_2$ and  $b_3$ are three bases. Then the closed segment from $p_1$ to $p_2$ is a block and the closed segment from $p_2$ to $p_3$ is another block.} \label{blocks}
\end{figure}
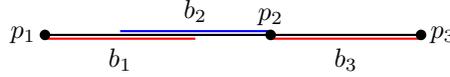

\begin{definition}
A \textbf{leaf} of $Y$ is an equivalence class of the equivalence relation on $Y$ generated by $x\sim x'$ if $\{x,x'\}$ is a vertical subset in $Y$.  
\end{definition}

The decomposition of $Y$ into leaves is a \textit{foliation} with a natural transverse measure, see \cite{ms88}. 
A Cantor-type self map on $Y$ can be constructed such that $Y$ is being laminated as well, see, for example in \cite{hatcher88}. So $Y$ is considered equipped with both measured lamination and measured foliation. We switch freely between these two terms for our advantage.

\begin{remark}
Let $Y$ be a union of  bands constructed by gluing bands to $\Gamma$ where $\Gamma$ is a disjoint union of finite compact simplicial trees. According to \cite[Lemma 6.1]{bfstableactions}, after applying a finite sequence of \textit{moves} on $Y$,  we may assume that $Y$ has the following property.

\vspace{1pc}
$(A1)$: \textit{Its underlying real graph $\Gamma$ is a disjoint union of edges. Each edge is either a block or meets no bands.}

\vspace{1pc}

We will make assumption $(A1)$ for all union of bands throughout the rest of this paper unless otherwise mentioned.
\end{remark}

\begin{definition} 
Let $Y$ be a union of bands with underlying real graph $\G$. For $z\in \G$ denote by $N_{\G}(z,\e)$ the closed $\e$-neighborhood of $z$ in $\G$. If $z\notin \G$ we take $N_{\G}(z,\e)$ to be empty. For a band $B=b \times I$ of $Y$, $z=(z, t)\in b\times I=B$, and $\e>0$, let $N_{B}(z,\e)$ be the closed horizontal segment $\{(x',t)\in B| |x-x'|\leq \e\}$. For $z\in Y$, define $N(z, \e)$ to be the union of $N_{\G}(z,\e)$ and of the $N_{B}(z,\e)$ over all $B$ containing $z$. A subset of $Y$ is \textbf{horizontal} if it is a subset of $\G$ or if it is contained in a single band and is horizontal there. Let $s, s'\subset Y$ be horizontal and let $p$ be a path in a leaf of $Y$. We say $s$ \textbf{pushes into $s'$ along $p$} if there is a homotopy $H$ of $s$ into $s'$ through horizontal sets such that 
\begin{itemize}
\item for every $z\in s$, $H(\{z\}\times I)$ is contained in a leaf, and 
\item there is $z_0\in s$ so that $p(t)=H(z_0, t)$ for all $t\in I$.
\end{itemize}
Given a horizontal $s$ and path $p$ in a leaf with $p(0)\in s$, we say that \textit{$s$ pushes along $p$} if there is a horizontal set $s'$ such that $s$ pushes into $s'$ along $p$. 

As subset $S$ of $Y$ is \textbf{pushing saturated} if given a path $p$ with $p(0)\in S$ and $\e>0$ such that $N(p(0),\e)$ pushes along $p$ then $p(1)\in S$. 
\end{definition}

\begin{definition}
Let $Y$ be a union of bands. A connected component of $Y$ is \textbf{minimal} if every pushing saturated subset contained in that component is dense in $Y$ and is \textbf{simplicial} if every leaf in that component is compact. 
Based on assumption (A1), we include in this description a special class of simplicial components, \textbf{naked real edges}, which are components of the real graph $\Gamma$ that meets no bands (or what we might call \textit{degenerate bands})
\end{definition}

A priori it is possible that a component of $Y$ is neither minimal nor simplicial. However, with a little extra restriction on leaves of $Y$, a component is either minimal or simplicial.  

\begin{prop}\cite[Proposition 4.8]{bfstableactions}\label{simormin}
Let $Y$ be a union of bands. There are only finitely many isotopy classes of compact leaves in $Y$. Suppose that no leaf of $Y$ has a subset that is proper, compact and pushing saturated. Then, the following holds. 
\begin{itemize}
\item Each component of $Y$ is either simplicial or minimal.
\item Each simplicial component of $Y$ is an $I$-bundle over some leaf in that component.
\item Only finitely many leaves of $Y$ contain a proper pushing saturated subset.
\end{itemize}
\end{prop}


\vspace{1pc}

From the construction, the fundamental group of a union of bands $Y$ is a free group. To study actions of a general group, relations are needed. Thus additional cells are attached to $Y$ to build a 2-complex as follows.

\begin{definition}\label{bandcomplex}
A \textbf{band complex} $X$ is a relative \textit{CW} $2$-complex based on a union of bands $Y$.  $X$ is obtained from $Y$ with $0, 1$, and $2$-cells attached such that 
\begin{itemize}
\item only finitely many (closed) cells of $X$ meet $Y$,
\item the $1$-cells of $X$ intersect $Y$ in a subset of $\Gamma$,
\item each connected component of the intersection of $\Gamma$ and a $2$-cell of $X$ is a point, and
\item each connected component of the intersection of a band of $Y$ and a $2$-cell of $X$ is vertical.
\end{itemize}
Intersections between attaching cells and $Y$ are call \textbf{attaching regions}.
\end{definition}
\vspace{1pc}

\begin{notation}
Here is a terminology convention. If $X$ is a band complex then $Y$ is always the underlying union of bands and $\G$ is always the underlying real graph. Similarly, the union of bands for $X'$ is $Y'$, etc. Further, $\overline{Y}$ denotes the union of bands obtained from $Y$ by omitting weight $0$ bands.
\end{notation}
\vspace{1pc}

A \textit{leaf} of a band complex $X$ is a leaf of $Y$.  
Similarly, a minimal or simplicial component of the lamination on $X$ (i.e. $Y$) is called a minimal or simplicial component of $X$ for short. The \textit{complexity} of $X$ is the complexity of $Y$. The transverse measure on $Y$ can then be integrated along a path in $X$. A \textit{generalized leaf} of $X$ is an equivalence class of points of $X$ under the equivalence relation $x\sim x'$ if there is a path in $X$ joining $x$ and $x'$ with measure $0$. $X^*\subset X$ is a \textit{band subcomplex} if $X^*$ is a band complex with underlying real graph $\Gamma^*\subset \Gamma$ and underlying union of bands $Y^*\subset Y$.

\vspace{1pc}
Let $G$ be a finitely presented group and $T$ be a $G$-tree. A band complex associated to $T$ can be obtained in the following way. 

\begin{definition}\label{resolution}
Let $X$ be a 2-complex whose fundamental group is $G$. A \textbf{resolution} for $T$ is a $G$-equivariant map $r: \widetilde{X}\to T$, where $\widetilde{X}$ is the universal cover of $X$, such that the image of a generalized leaf of $\widetilde{X}$ is a point and $r$ embeds the lifts of bases of bands in $X$. We say that $X$ a \textbf{resolving complex}.
\end{definition}

\begin{prop} \cite[Proposition 5.3]{bfstableactions} \label{treehasresolution}
Let $G$ be a finitely presented group. Every $G$-tree $T$ has a resolution.
\end{prop}

For a given $G$-tree $T$, let $X$ a band complex resolving $T$. 
The goal of the Rips machine is to put this arbitrary band complex $X$ into some normal form, further through which one can study structures of $T$ and $G$. There is a list of six moves (M0)-(M5) that can be applied to a band complex described in \cite[section 6]{bfstableactions}. These moves are elementary homotopic moves with respect to the underlying measured lamination. If a band complex $X'$ is obtained from a band complex $X$ by a sequence of these moves, the following holds.
\begin{itemize}
\item There are maps $\phi: X\to X'$ and $\psi: X'\to X$ that induces an isomorphism between fundamental groups and preserve measure.
\item If $f:\widetilde{X}\to T$ is a resolution, then the composition $f\tilde{\psi}:\widetilde{X}'\to T$ is also a resolution, and if $g: \widetilde{X}'\to T$ is a resolution, then so is $g\tilde{\phi}: \widetilde{X} \to T$.
\item $\phi$ and $\psi$ induce a $1-1$ correspondence between the minimal components of the laminations on $X$ and $X'$
\end{itemize}

\vspace{1pc}
We now briefly describe four of these six moves that are extensively used in later sections (following their original labels in \cite{bfstableactions}) .

$\textbf{(M2)}$ \textbf{Subdivide a band}. Let $B=b\times I$ be a band, and let $x$ be a point in the interior of $b$.  
We obtain $Y'$ by splitting $B$ down the vertical fiber $\{ x \} \times I$. There are now two fibers over $x$. To obtain $X'$, first attach a new $2$-cell $D$ to $Y'$ along the two copies of $\{x\}\times I$. Then attach the cells of $X$ to $X'$ so that collapsing $D$ in $X'$ recovers $X$. $D$ is called a \textit{subdivision cell}. 
\vspace{1pc}

$\textbf{(M3)}$ \textbf{Cut $\Gamma$}. Let $z$ be a point of $\Gamma$ that does not meet the interior of any base. $\Gamma'$ is $\Gamma$ split open at $z$. Let $Z$ be the set of points in $\Gamma'$ which lies over $z$. Attach to $\Gamma'$ the cone over $Z$. This added a collection of new $1$-cells and a new $0$-cell in $X'$. To complete this move, attach bands and cells to $X'$ so that collapsing the above cone recovers $X$.

\vspace{1pc}

$\textbf{(M4)}$ \textbf{Slide}. Let $B=b\times I$ and $C=c\times I$ be distinct bands in $X$. Suppose that $f_b(b)\subset f_c(c)$. Then a new band complex $X'$ can be created by replacing $f_b$ by $f_{dual(c)}\circ \delta_{C}\circ f^{-1}_c\circ f_b$. This is a slide of $b$ across $C$ from $c$ to $dual(c)$. We say that $c$ is the \textit{carrier}, and b is \textit{carried}. See Figure~\ref{Fgslide}.

\begin{figure}[h!]
   \centering
      \begin{tikzpicture}[thick, scale=0.8]
      \footnotesize
\draw[fill=lightgray] (-4,0)--(-2,0)--(0,2)--(-2,2)--(-4,0);
\draw[fill=lightgray] (-3.5,0)--(-2.5,0)--(-2.5,2)--(-3.5,2)--(-3.5,0);
\node [above] at (-2,0.8) {$C$};
\node [above] at (-3,0.8) {$B$};
\draw  [->](0,1) -- (3,1);
\draw[fill=lightgray] (4,0)--(6,0)--(8,2)--(6,2)--(4,0);
\node [above] at (6,0.8) {$C$};
\draw [fill=lightgray] (6.5,2)--(7.5,2)--(7.5,4)--(6.5,4)--(6.5,2);
\node [above] at (7,2.8) {$B$};
\end{tikzpicture}
\caption{\footnotesize Slide $B$ across $C$.} \label{Fgslide}

\end{figure}
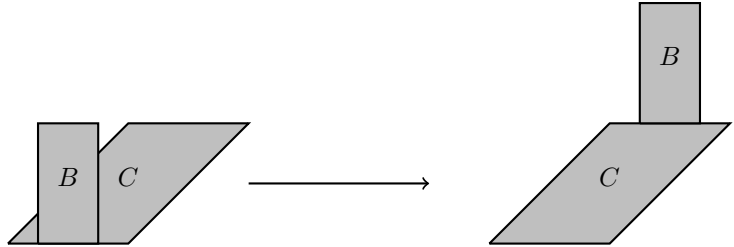

\begin{definition}\label{freesubarc}
Let $B=b \times I$ be a band of $Y$. A subarc $J$ of a base $b$ is \textbf{free} if either 
   \begin{itemize}
      \item $b$ has weight $1$ and the interior of $J$ meets no other base of positive weight, see Figure \ref{freesubarcp}, or
      \item $b$ has weight $\frac{1}{2}$, the interior of $J$ does not contain the midpoint of $b$,  and $b$ and $\text{dual}(b)$ are the only positive weight bases that meet the interior of $J$.
   \end{itemize}
\end{definition}

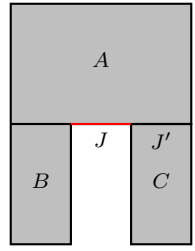
\begin{figure}[h!]
\begin{tikzpicture}[thick, scale=0.8]
\tiny
\draw[fill=lightgray] (0,0)--(3,0)--(3,2)--(0,2)--(0,0);
\draw[fill=lightgray] (0,0)--(1,0)--(1,-2)--(0,-2)--(0,0);
\draw[fill=lightgray] (2,0)--(3,0)--(3,-2)--(2,-2)--(2,0);
\draw[red][thick](2,0)--(1,0);
\draw[thick](2,0)--(3,0);
\node [above] at (1.5, 0.8) {$A$};
\node [above] at (0.5, -1.2) {$B$};
\node [above] at (2.5, -1.2) {$C$};
\node [below] at (1.5, 0) {$J$};
\node [below] at (2.5, 0) {$J'$};

\end{tikzpicture}
\caption{\footnotesize $J$ is a free subarc, whereas $J'$ is not} \label{freesubarcp}
\end{figure}

$\textbf{(M5)}$ \textbf{Collapse from a free subarc}.
Let $B=b\times I$ be a band, and let $J\subset b$ be a free subarc. First use (M2) to subdivide weight $0$ bands meeting $J$, until a base of any such band meeting $J$ is contained in $J$.  In doing this, we may have introduced new subdivision cells into $X$. Then use (M4) to slide the weight $0$ bands that meet $J$ over $B$ to $dual(b)$. Call the result $X^*$.  New subdivision cells introduced by subdivision become annuli in $X^*$, so we call them \textbf{subdivision annuli}. Further, relative $1$- and $2$-cells whose image intersect interior of $J\times[0,1)$, are naturally homotoped upwards to $dual(b)$. 

If $b$ has weight $1$, collapse $J\times I$ to $dual(J)\cup (\partial J\times I)$ to obtain $X'$, see Figure~\ref{fgcollapse}. Typically, the band $B$ will be replaced by two new bands. But if $J$ contains one or both endpoints of $b$, say $x$, then we also collapse $\{x\} \times I$. Thus $B$ is replaced by $1$ or $0$ band. 

If $b$ is of weight $\frac{1}{2}$, we may subdivide $B$ over the end-point of $J$ nearest the midpoint of $b$ such that $J$ is contained in a band of weight $1$ (\cite[Lemma 6.5]{bfstableactions}). Then we can collapse from $J$ as before.

\begin{figure}[h]
   \centering
      \begin{tikzpicture}[thick, scale=0.8]
      \footnotesize
\draw [fill=lightgray](-4,0)--(0,0)--(0,2)--(-4,2)--(-4,0);
\draw (0,-0.5)--(-4,-0.5);
\draw (-4,0) to [out=180,in=180] (-4,-0.5);
\draw (-4,0) to [out=0,in=0] (-4,-0.5);
\draw (0,0) to [out=0,in=0] (0,-0.5);
\draw [dashed](0,0) to [out=180,in=180] (0,-0.5);
\draw[ultra thick][dashed](-3,0)--(-1,0);
\node [above] at (-2,0) {$J$};
\node [above] at (-2,-0.6) {$C$};
\draw[dashed](-3,0)--(-3,2)--(-1,2)--(-1,0)--(-3,0);
\draw[ultra thick] (-3,2)--(-1,2);
\node [above] at (-2,2) {$dual(J)$};

\draw  [->](2,1) -- (3,1);
\draw [fill=lightgray](7,2)--(7,0)--(8,0)--(8,2)--(4,2)--(4,0)--(5,0)--(5,2);
\draw[thick] (5,2)--(7,2);
\draw (4,-0.5)--(5,-0.5);
\draw (7,-0.5)--(8,-0.5);
\draw (5,2.5)--(7,2.5);
\draw (4,0) to [out=180,in=180] (4,-0.5);
\draw (4,0) to [out=0,in=0] (4,-0.5);
\draw [red](5,0) to [out=0,in=0] (5,-0.5);
\draw [red][dashed](5,0) to [out=180,in=180] (5,-0.5);
\draw [blue](7,0) to [out=180,in=180] (7,-0.5);
\draw [blue](7,0) to [out=0,in=0] (7,-0.5);
\draw (8,0) to [out=0,in=0] (8,-0.5);
\draw [dashed](8,0) to [out=180,in=180] (8,-0.5);

\draw [red](5,2.5) to [out=180,in=180] (5,2);
\draw [red](5,2.5) to [out=0,in=0] (5,2);
\draw [blue][dashed](7,2.5) to [out=180,in=180] (7,2);
\draw [blue](7,2.5) to [out=0,in=0] (7,2);
\end{tikzpicture}
\caption{\footnotesize Collapse from a free subarc $J$: first use (M2) to subdivide the weight $0$ band $C$ into three bands and slide the middle one across $B$. Subdivison annuli are attached between two red loops and two blue loops. Then collapse $J\times I$ to $dual(J)\cup (\partial J\times I)$ } \label{fgcollapse}

\end{figure}
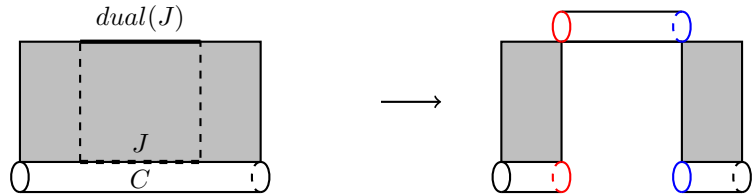

\vspace{1pc}
In this paper, we will need a few additional moves listed below.

\textbf(M6) \textbf{Attach a band}: Glue $b\times I$ to $X$ via a measure-preserving map from $b \times \partial I$ to its underlying real graph $\Gamma$ that transverse to the measured lamination.
\vspace{1pc}

\textbf(M7) \textbf{Adding an arc to the real graph}: Add an extra segment $c$ of the real line to $\Gamma$. A special case is enlarge $\Gamma$ by gluing one end point of $c$ to an end point of $\Gamma$.
 \vspace{1pc}  

\textbf(M8) \textbf{Attach a disk} \footnote{This is a general version of (M$0'$) in \cite{bfstableactions}. }: Attach a 2-disk $D$ to $X$ along a loop in the relative $1$ skeleton of $X$ (i.e. the union of bands $Y$ and $1$-cells in $X$) whose measure is zero. 
\vspace{1pc}
\begin{remark}\label{newmoves}
Note that unlike the moves in \cite{bfstableactions}, (M6)-(M8) may change the fundamental group of $X$. But we never perform such moves alone. For example, right after a (M6) - attach a band, we always perform some (M8) - attach a disk (multiple times if needed) to make sure that the resulting band complex $X'$ has the same fundamental group as $X$ and remains resolving the same tree as $X$.
\end{remark}

\begin{prop}\cite[Lemma 6.1]{bfstableactions}\label{nofault}
By applying a sequence of moves to a band complex $X$, we may arrange that no leaf of the complex has a subset that is proper, compact, and pushing saturated.
\end{prop}

\vspace{1pc}

Proposition~ \ref{simormin} and Proposition ~\ref{nofault} imply the following. After possibly a sequence of moves, we may always assume that each component of the measured lamination on $X$ is either simplicial or minimal. The Rips machine is designed to study one minimal component at a time. It consists of two processes (made of sequence of moves): Process~I and Process~II. Neither process increases complexity. 

Roughly speaking, each step of Process~I is a (M5)-collapse a free subarc (see details in Section~\ref{thinreview}); each step of Process~II is a combination of a (M4)-slide and a (M5)-collapse a free subarc (see details in Section~\ref{processII}). 

Fix a minimal component $Y_0$ of $Y$. First, Process I is successively applied until no base of $Y_0$ has a free subarc. It is possible that Process I continues forever. In this case, as a machine output, we say that $Y_0$ is of \textbf {thin type}  (also called \textbf{Levitt} or \textbf{exotic type}). If no base has a free subarc then Process II is successively applied until the complexity of $Y_0$ decreases then go back to Process I. It is possible that Process II continues forever. In this case, as a machine output, $Y_0$ is either of \textbf{surface type} or of \textbf{toral type} (also called \textbf{axial type}) (depending on $Y_0$'s local structure).
Thus eventually only one of the processes is applied, and it continues forever. After knowing the type of $Y_0$, we continue by choosing another component in $Y$.

\begin{definition}\label{standardform}
In applying the Rips machine to $Y_0$, we will obtain an infinite sequence $Y_0,Y_1\dots$ where each component is obtained from the previous by a process. Further, there exists an integer $N$ such that only a unique process (either Process~I or Process~II) is applied after $Y_N$ and $Complexity(Y_n)=Complexity(Y_{n+1})$ for all $n>N$. For $n>N$, we say that $Y_n$ is in a \textbf{standard form}. A band complex is \textbf{standard} if its every minimal component is standard. 
\end{definition}

\begin{definition}\label{index}
   Given a union of bands $Y$, let $q\in Y$ be a point and $l_{q}\subset Y$ be the leaf containing $q$. Each component $d_p$ of $l_q-{q}$ is called a \textbf{direction} of $Y$ at $q$. A direction is \textbf{infinite} if the corresponding component of a leaf is infinite. In particular, for every $q\in \Gamma$, a band in $Y$ containing $q$ determines a unique direction at $q$. Denote the \textbf{direction set} at $q$ by $T_{q}Y$. 
   
   For each $q\in \overline{Y}$, 
   we define its \textbf{index} by $$i_{Y}(q)=\# \{\text{infinite directions at $q$}\}-2.$$ The \textbf{limit set} $\Omega_{Y}$ of $Y$ is the set of points in $\overline{Y}$ whose index is at least zero. Its intersection with the real graph (i.e. $\Omega_{Y}\cap \Gamma$) is called the \textbf{limit graph}.
\end{definition}


It is proved in \cite[section 8]{bfstableactions} that each type of component in $Y$ can be characterized in the following fashion.

\begin{prop}\label{character} 
Let $X$ be a band complex and $Y$ be the underlying union of bands. Application of the Rips machine puts $X$ into a \textit{standard form}. Each component $Y_0$ of $Y$ belongs to one of the four types: \textit{simplicial, surface, toral and thin}. 
Moreover, for a fixed component $Y_0$, let $\Omega_{Y_0}$ be its limit set (i.e. $\Omega_{Y}\cap Y_0$) and $\Gamma_{Y_0}=\Gamma\cap Y_0$ be its underlying real graph. Then 

\begin{enumerate}
\item $Y_0$ is a simplicial component if and only if $\Omega_{Y_0}$ is empty;

\item $Y_0$ is a toral component with rank $n>2$ if and only if $\Omega_{Y_0}$ contains infinitely may points of positive index;

\item $Y_0$ is a surface component if and only if $i_Y(q)=0$ for every interior point $q$ in $Y_0$ and the closure of $\Omega_{Y_0}$ is $Y_0$;

\item $Y_0$ is a thin component if and only if there are finitely many points in $\Omega_{Y_0}$ have positive index and $\Omega_{Y_0}\cap \Gamma_{Y_0}$ is a dense $G_{\delta}$ set in $\Gamma_{Y_0}$.

\end{enumerate}

\end{prop}

Associated to a band complex is a $GD$ which is a generalization of the $GAD$ defined in \cite{bfnotes}.

\begin{definition}
   A \textit{generalized decomposition}, or $GD$ for short, is a graph of groups presentation $\Delta$ where some vertices have certain extra structure. Namely, the underlying graph is bipartite with vertices in one class called \textit{rigid} and vertices in the other called \textit{foliated}. Further each foliated vertex has one of four types: \textit{simplicial, toral, surface} and \textit{thin}.
\end{definition}

A band complex $X$ is naturally a graph of spaces where vertex spaces are components of $Y$ (with attaching $2$-cells that is glued along its leaves) and the closures of connected components of the complements of $Y$. An edge corresponds to a connected component of the intersection of the closures of sets defining two vertices. The $GD$ associated to $X$, denoted by $\Delta(X)$, is the graph of groups decomposition coming from this graph of spaces.

\section{Overview of relative Rips machine}\label{RRM}

Given a group acting on an $\R$-tree, one may ask how its subgroups act on their corresponding minimal subtrees. In Section~\ref{RRM} and Section~\ref{process}, we will construct a relative version of the Rips machine to study such subgroup actions. The relative Rips machine takes as input a \textit{pair} of band complexes (see Definition~\ref{pair}): two band complexes $X$ and $X'$ with a \textit{morphism} $\i: X\to X'$, denoted by $(X\overset{\i}{\to}X')$. The goal of this machine is to understand the relation between structures of $X$ and $X'$ by putting both $X$ and $X'$ into some normal forms simultaneously. As with the Rips machine, this relative Rips machine studies one pair of components $(Y_0\overset{\i_0}{\to}Y'_0)\subset (Y\overset{\i}{\to}Y')$ at a time. To be more precise, for fixed minimal component $Y'_0$ of $Y'$, let $Y_0$ (possibly simplicial) be a component of $Y$ such that $\i(Y_0)\subset Y'_0$ and $\i_0$ be the restriction of $\i$ on $Y_0$. In general, for a fixed $Y'_0$, there are finitely many choices for $Y_0$. To simplify the notation, we will first work on the special case where there is a unique choice for $Y_0$. Then in Section~ \ref{generalcase}, we will justify the machine described under this special setting for the general case.

\begin{definition}\label{bandmorphism}
Let $X$ and $X'$ be band complexes.  A map between their underlying real graphs $\i: \Gamma\to \Gamma'$ is a \textit{morphism} if every edge of $\Gamma$ has a finite subdivision such the restriction of $\i$ to each segment of the subdivision is an isometry. A map between their unions of bands $\i:Y\to Y'$ is a \textit{morphism} if it restricts to a morphism of real graphs and takes each band of $Y$ homeomorphically to a subband of $Y'$. Finally, a relative cellular map $\i: X\to X'$ is called a \textbf{morphism} if its restriction to the underlying union of bands is a morphism. An inclusion of a band subcomplex into a band complex is an example of morphism.


\end{definition}

\begin{definition}
Let $Y$, $Y'$ be two unions of bands and $\i: Y\to Y'$ be a morphism. For a point $q\in Y$, $\i$ is \textbf{locally injective} near $q$, if there exists a neighborhood of $q$ in $Y$ such that the restriction of $\i$ to that neighborhood is injective. $\i$ is an \textbf{immersion} if $\i$ is locally injective near any point $q\in Y$. Further let $p=\i(q)\in \i(Y)\subset Y'$. $\i$ is \textbf{locally surjective} near $q$, if for any neighborhood $U_q$ of $q$, there is a neighborhood $U_p$ of $p$ with the property that $U_p \subset \i(U_q)$. $\i$ is a \textbf{submersion} if $\i$ is locally surjective near any point $q\in Y$. We say $\i$ is a \textbf{local isomorphism} if $\i$ is an immersion and also a submersion. 

\end{definition}

\begin{definition}\label{pair}
 Given two band complexes $X$ and $X'$, we say they form a \textbf{pair} if there is a morphism $\i: X\to X'$, denoted by $(X\overset{\i}{\to}X')$. Correspondingly, $(Y\overset{\i}{\to}Y')$ is a \textit{pair of unions of bands}. $(Y_0\overset{\i_0}{\to}Y_0')$ is a \textit{pair of components} where $\i_0=\i|_{Y_0}$, $Y_0\subset Y$ and $Y'_0\subset Y'$ are components with the property that $\i_0(Y_0)\subset Y'_0$. 
\end{definition}

\begin{definition}\label{pairoftrees}
  Let $H<G$ be two finitely presented groups, $T_G$ be a $G$-tree, $T_H\subset T_G$ be a minimal $H$-subtree. $(T_H\hookrightarrow T_G)$ is called a \textbf{pair of trees}. Further let $(X_H\overset{\i}{\to}X_G)$ be a pair of band complexes. We say $(X_H\overset{\i}{\to}X_G)$ resolves $(T_H\hookrightarrow T_G)$ if there exist resolutions $r_G:\widetilde{X}_G\to T_G$ and $r_H:\widetilde{X}_H\to T_H$ such that the following diagram commutes.
      
 $$\begindc{\commdiag}[45]
\obj(0,0)[XH]{$\widetilde{X}_H$}
\obj(1,0)[TH]{$T_H$}
\obj(0,1)[XG]{$\widetilde{X}_G$}
\obj(1,1)[TG]{$T_G$}
\mor{XH}{TH}{$r_H$}
\mor{XH}{XG}{$\tilde{\i}$}
\mor{TH}{TG}{}[\atright,\injectionarrow]
\mor{XG}{TG}{$r_G$}
\enddc  $$

   \end{definition}

\begin{remark}
Similarly as for all moves in the Rips machine, all moves that will be described for pairs of band complexes obey the following rule. 
If $(X\overset{\i}{\to}X')$ and $(X^*\overset{\i}{\to}{X'}^*)$ are related by a move and $(X\overset{\i}{\to}X')$ resolves $(T_H\hookrightarrow T_G)$ then $(X^*\overset{\i}{\to}{X'}^*)$ also resolves $(T_H\hookrightarrow T_G)$. 
\end{remark}

We now discuss some further assumptions that one can make for a pair $(X\overset{\i}{\to}X')$. 
 
 \vspace{1pc}

Given a band complex $X$, according to \cite[Lemma 6.1]{bfstableactions}, we may always arrange that no leaf of $X$ has a subset that is proper, compact, and pushing saturated by a sequence of moves. This further allows us to make the assumption that its underlying union of bands $Y$ is a disjoint union of components. In fact, we can arrange this assumption for a pair of band complexes $(X\overset{\i}{\to}X')$ as well. There are only finitely many proper, compact, pushing saturated subsets of leaves in $X$ and $X'$ (since each of such subsets contains at least one vertical fiber in the boundary of $Y$). Firstly, we may remove all such ``bad'' subsets of leaves in $X$ by (M2) and (M3) as in \cite[Lemma 6.1]{bfstableactions}. Denote the resulting band complex by $X^*$. A morphism $\i^*:X^*\to X'$ is defined as the following. On the level of union of bands, $\i^*$ is the composition of the inclusion $Y^*\hookrightarrow Y$ and $\i:Y\to Y'$ which remains an morphism. $\i^*$ maps subdivision annuli newly created by (M2) and cones in $\Gamma(X^*)$ newly created by (M3) to the $\i$-images of the corresponding split vertical fibers and split vertices. And $\i^*$ remains the same as $\i$ on all the other attaching cells. Secondly, we may then remove such ``bad'' subsets in $X'$ by (M2) and (M3) to obtain $X'^*$ and these moves are disjoint from $\i^*(X^*)$ since leaves of $X^*$ no long contain any ``bad'' subsets. $(X^*\overset{\i}{\to}X'^*)$ is a pair satisfying the desired assumption. Hence, from now on we will assume that for a pair of band complexes $(X\overset{\i}{\to}X')$, the following holds.
 
   \begin{itemize}
      \item $(A2)$: No leaf of $X$ or $X'$ has a subset that is proper, compact, and pushing saturated. In particular, their underlying unions of bands $Y$ and $Y'$ are unions of disjoint components.
   \end{itemize}



\vspace{1pc}

In Definition ~\ref{bandmorphism}, a morphism $\i:X\to X'$ takes each band of $Y$ homeomorphically to a subband of $Y'$. In general, it may not hold on the level of generalized bands. The following lemma allows us to work only with generalized bands.

\begin{prop}\label{morphsimongb}
Let $(Y\overset{\i}{\to}Y')$ be a pair of union of bands. $\i$ can be converted into a morphism that takes each generalized band of $Y$ homeomorphically to a sub-generalized band of $Y'$ by a finitely sequence of moves.
\end{prop}

\begin{proof}
 Let $\BB_Y$ be a generalized band of $Y$ and its image in $Y'$ be $\widehat{\BB}_Y=\i(\BB_Y)$. 
 
 (1) We may assume that $\widehat{\BB}_Y$ is contained in a single generalized band $\BB_{Y'}$ in $Y'$. Otherwise $\widehat{\BB}_Y$ is contained in the union of two or more consecutive generalized bands 
   in $X$. In this case, we can horizontally (transverse to the foliation) subdivide $\BB_Y$ into several generalized bands such that the image of each of these is contained in a single generalized band, see Figure~\ref{Fs1}; 
   
   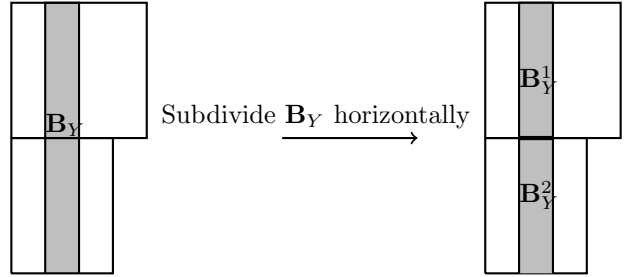
\begin{figure}[h]
     \centering 
     \begin{tikzpicture} [thick, scale=0.9]
     \footnotesize
\draw [fill=lightgray] (0.5,-2)--(0.5,2)--(1, 2)--(1,-2);
\draw (0,-2)--(0,0)--(1.5,0)--(1.5,-2)--(0,-2);
\draw (0,0)--(2,0)--(2, 2)--(0,2)--(0,0);

\node [below] at (0.8, 0.5) {$\BB_Y$};

\draw [->] (4,0)--(6,0);
\node [above]  at (4.5,0) {Subdivide $\BB_Y$ horizontally};

\draw (7,-2)--(7,0)--(8.5,0)--(8.5,-2)--(7,-2);
\draw (7,0)--(9,0)--(9, 2)--(7,2)--(7,0);
\draw [fill=lightgray] (7.5,-2)--(7.5,2)--(8, 2)--(8,-2);
\draw [ultra thick] (7.5,0)--(8,0);
\node [above] at (7.8,0.5) {$\BB_Y^1$};
\node [above] at (7.8,-1.2) {$\BB_Y^2$};

\end{tikzpicture}

\caption{\footnotesize The figure shows the case where $Y$ is identified with its image in $Y'$, i.e. $Y\hookrightarrow Y'$.
Blank rectangles represent generalized bands in $Y'$. The gray parts are generalized bands in $Y$.} \label{Fs1}
   \end{figure}

  (2) We may assume that there is no two or more consecutive generalized bands $\BB_Y^1, \BB_Y^2 \dots$ of $Y$ map into the same generalized band of $Y'$. Otherwise, let $b_Y^i$ and $dual(b_Y^{i})$ be the bottom and top bases of $\BB_Y^i$, $i=1,2,\dots$. Then $dual(b_Y^{i-1}) \cup b_Y^i - dual(b_Y^{i-1}) \cap b_Y^i$ is a union of free subarcs. After collapsing these free subarcs, $\BB_Y^1, \BB_Y^2,\dots$ form a new longer generalized band, whose image in $Y'$ is a single sub-generalized band. See Figure~\ref{Fs2};
   
   \begin{figure}[h]
   \centering
      \begin{tikzpicture}[thick, scale=0.9]
      \footnotesize
      \draw [fill=lightgray] (0.3,-2)--(0.3,0)--(1.3, 0)--(1.3,-2);
\draw [fill=lightgray] (0.7,0)--(0.7,2)--(1.7, 2)--(1.7,0);
\draw (0,-2)--(0,0)--(2,0)--(2,-2)--(0,-2);
\draw (0,0)--(2,0)--(2, 2)--(0,2)--(0,0);

\draw [ultra thick] (0.3,0)--(0.7,0);
\node [above] at (0.5,0) {$J_1$};
\draw [ultra thick] (1.3,0)--(1.7,0);
\node [below] at (1.5,0) {$J_2$};
\node [below] at (0.8, -0.5) {$\BB_Y^1$};
\node [above] at (1.2, 0.7) {$\BB_Y^2$};

\draw [->] (4,0)--(6,0);
\node [above]  at (4.5,0) {Collapse $J_1$ and $J_2$};

\draw [fill=lightgray] (7.7,-2)--(7.7,2)--(8.3, 2)--(8.3,-2);
\draw (7,-2)--(7,0)--(9,0)--(9,-2)--(7,-2);
\draw (7,0)--(9,0)--(9, 2)--(7,2)--(7,0);

\node [below] at (8, 0.5) {$\BB_Y$};

\end{tikzpicture}
\caption{\footnotesize This is the case where $Y\hookrightarrow Y'$. In this picture, $dual(b_Y^{1}) \cup b_Y^2 - dual(b_Y^{1}) \cap b_Y^2=J_1\cup J_2$.} \label{Fs2}
\end{figure}
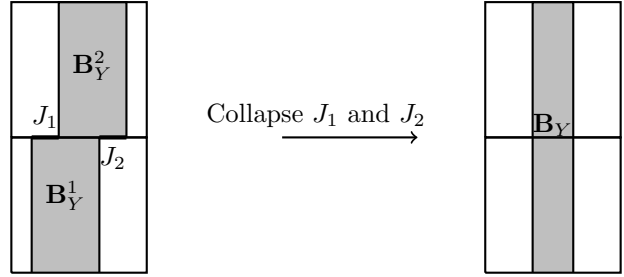

(3) Finally, we may assume that $l(\widehat{\BB}_Y)=l(\BB_{Y'})$ i.e. $\widehat{\BB}_Y$ is a sub-generalized band. Otherwise $l(\widehat{\BB}_Y)<l(\BB_{Y'})$. One of the bases of $\widehat{\BB}_Y$ is a type I free subarc. This implies that one of the bases of $\BB_Y$ is a free subarc of type I and so the whole $\BB_Y$ can be collapsed. See Figure ~\ref{Fs3}.  
   
   \begin{figure}[h]
   \centering
   \begin{tikzpicture}[thick, scale=0.8]
   \footnotesize
   \draw [fill=lightgray] (0.5,-2)--(0.5,0)--(1.5, 0)--(1.5,-2);
\draw (0,-2)--(0,1)--(2,1)--(2,-2)--(0,-2);

\node [above] at (1, -1) {$\BB_Y$};

\draw [->] (4,0)--(6,0);
\node [above]  at (4.5,0) {Collapse $\BB_Y$};

\draw (7,-2)--(7,1)--(9,1)--(9,-2)--(7,-2);

\end{tikzpicture}
\caption{\footnotesize This shows the case where $Y\hookrightarrow Y'$ .}\label{Fs3}      
   \end{figure}
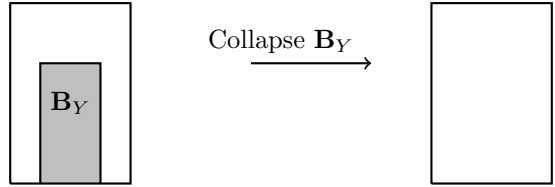

The number of generalized bands in $Y'$ is bounded. So after finitely many moves, $\i$ maps every generalized band in $Y$ to a sub-generalized band in $Y'$. 
\end{proof}

Proposition ~\ref{morphsimongb} allows us to make the following assumption. In the rest sections of this paper, a band always meant to be a generalized band unless otherwise stated.

\begin{itemize}
      \item $(A3)$: A morphism between two band complexes always represents the morphism it induced on the level of generalized bands.
\end{itemize}

\vspace{1pc}
\begin{definition}\label{fold}
   Let $Y$ be an union of bands and $B_1$, $B_2$ be two distinct bands (both have length $1$, i.e. the original bands, not the generalized bands) in $Y$.  If a base of $B_1$ is identified with a base of $B_2$, we say that these two bands have a \textbf{common base}, denoted it by $b$. Further suppose that the bases of $B_1$ and $B_2$ other than $b$ are either identified or contained in different blocks. Let $\phi: B_1\to B_2$  be the linear homeomorphism fixing $b$. Define $``\sim"$ be the equivalence relation on $Y$ such that $x\sim \phi(x)$, for all $x\in B_1$. The quotient space $Y/\sim$ is still an union of bands. Call the quotient map $f:Y\twoheadrightarrow Y/\sim$ a \textbf{fold of bands}.

In general, if a base $b_1$ of $B_1$ overlaps a base $b_2$ of $B_2$ ($B_1$ \textit{overlaps} $B_2$ for short), i.e. $b_1\cap b_2=o\neq \emptyset$, and further the dual positions of $o$ in $dual(b_1)$ and in $dual(b_2)$ are either identified or contained in different blocks (see Figure ~\ref{F7}), we may fold the subbands \footnote{Here``subband" is abused for the case where $o$ is a single point} determined by the common segment $o$ in $B_1$ and $B_2$ (the shaded parts in Figure~\ref{F7}). Moreover, we may also fold two generalized bands with a common base or an overlap if they have the same length.
\end{definition}

\begin{figure}[h]
   \centering
\begin{tikzpicture}[thick, scale=0.85]
\footnotesize
\draw (0,-3)--(0,0)--(2,0)--(2,-3)--(0,-3);
\node [left] at  (1.2,-1.8) {$B_2$};
\draw (1,0)--(1,3)--(3, 3)--(3,0)--(2,0);
\node [right] at  (2,1.4) {$B_1$};
\draw [pattern=north west lines] [dashed](1,0)--(1,3)--(2,3)--(2,0);
\node [right] at (3, 0.1) {$b_1$};
\draw [red](1,0.05)--(3,0.05);
\draw[ultra thick] (1,0)--(2,0);
\node [below] at  (1.5,0) {$o$};

\draw [dashed][pattern=north west lines] (1,0)--(1,-3)--(2,-3)--(2,0)--(1,0);
\node [left] at (0, -0.1) {$b_2$};
\draw [blue](0,-0.05)--(2,-0.05);

\draw [->] (5,0)--(8,0);
\node [above]  at (6.5,0) {Identify shaded parts };
\node [below]  at (6.5,0) {in $B_1$ and $B_2$ to obtain $B_3$};

\draw (9,-1.5)--(9,1.5)--(12,1.5)--(12,-1.5)--(9,-1.5);
\draw [dashed][pattern=north west lines] (10,-1.5)--(10,1.5)--(11, 1.5)--(11,-1.5)--(10,-1.5);
\draw [ultra thick] (10,-1.5)--(11,-1.5);
\node [below] at  (10.5,-1.5) {$o$};
\node [left] at (10,0){$B_3$};

\end{tikzpicture}
\caption{\footnotesize$B_1$ and $B_2$ are two bands in $Y$. Their bases overlap at a segment $o$. Fold the subbands in $B_1$ and $B_2$ determined by $o$ to obtain a new band $B_3$.}\label{F7}
\end{figure}
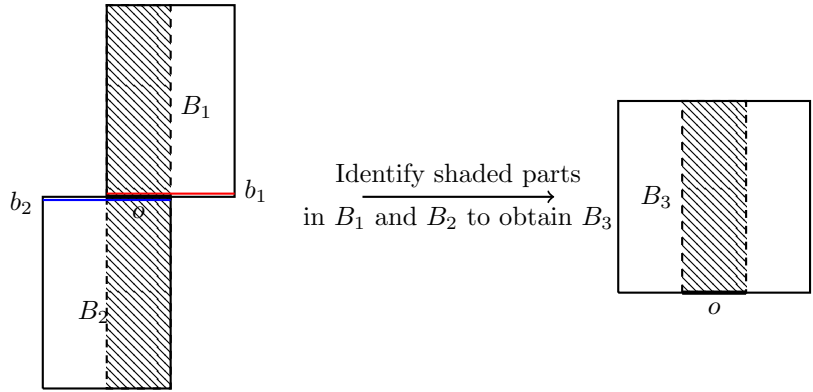

\begin{remark}
Recall we assume that a given union of bands $Y$ has the structure constructed in Definition~\ref{bandcomplex}, in particular $(A1)$ is satisfied. Roughly speaking, for a given pair of band complexes $(X\overset{\i}{\to}X')$ (with its underlying pair of unions of bands $(Y\overset{\i}{\to}Y')$), we will fold $Y$ according to the morphism $\i:Y\to Y'$ to get a new union of bands $Y^1$ ($=Y/\sim$) which is ``closer" to $Y'$. In general, the real graph $\Gamma_{Y^1}$ of $Y^1$ obtained after folding is a simplicial forest (does not satisfies $(A1)$ ). To fix this, we could do subdivision. However, the problem is that we may have to fold infinitely many times and the underlying real graph of the limiting union of bands may not be simplicial anymore. To bypass this, whenever $\Gamma_{Y^1}$ fails $(A1)$, we will replace $(Y^1\overset{\i^1}{\to}Y')$ by a new pair $(Y^{1*}\overset{\i^{1*}}{\to}Y')$ resolving the same pair of trees and in addition,$\Gamma_{Y^{1*}}$ satisfies $(A1)$. This replacement procedure is discussed in Proposition ~\ref{wideoverlap}.   
\end{remark}

 \begin{definition}
   
   Let $\i: Y\to Y'$ be a morphism between two unions of bands where $\Gamma_{Y}$ is a simplicial forest and $\Gamma_{Y'}$ is a union of edges and $\BB^1_Y,\BB^2_Y$ be two generalized bands in $Y$ that have an overlap. Let the overlap be $o=b^1_Y\cap b^2_Y$ where $b^1_Y\subset \BB^1_Y$ and $b^2_Y\subset \BB^2_Y$ are bases. Denote $\i(b^1_Y)\cap\i(b^2_Y)$ by $\widehat{o}$. $o$ and $\widehat{o}$ are segments. We say the overlap between $\BB^1_Y$ and $\BB^2_Y$ is \textbf{wide} if $o$ and $\widehat{o}$ have the same length. 
\end{definition}

\begin{prop}\label{wideoverlap}
Let $(X\overset{\i}{\to}X')$ be a pair of band complexes where $\Gamma_X$ is a simplicial forest. We may replace $(X\overset{\i}{\to}X')$ by a new pair $(X^*\overset{\i^*}{\to} X')$ such that it resolves the same pairs of trees as $(X\overset{\i}{\to}X')$ does and the restriction of $\i^*$ to the underlying real graphs $\i^*: \Gamma_{X^*}\to \Gamma_{X'}$ is an immersion. In particular, all overlaps between bands are wide.
\end{prop}

\begin{proof}
The restriction of $\i$ to the underlying real graphs $\i: \Gamma_{X}\to \Gamma_{X'}$ is a morphism between graphs, which can be realized as composition of finitely many folds (for graphs) and an immersion \cite{bfbounding}. Let $X^*$ be the resulting band complex obtained from $X$ by folding $\Gamma_X$ according to $\i$. Let $\i^*: X^*\to X'$ be the induced morphism. $(X^*\overset{\i}{\to} X')$ is then a new pair of band complexes with the desire property. 
\end{proof}

\begin{prop}\label{maimmersion}
   Let $(X\overset{\i}{\to}X')$ be a pair of band complexes where $\Gamma_X$ is a simplicial forest. After a finite sequence of folds between graphs and folds between unions of bands, we may replace $(X\overset{\i}{\to}X')$ by a new pair $(X^*\overset{\i^*}{\to} X')$ such that it resolves the same pairs of trees as $(X\overset{\i}{\to}X')$ does and $\i^*: Y^*\to Y'$ is an immersion.
\end{prop}

\begin{proof}
We may assume that $\i:\Gamma_X\to \Gamma_{X'}$ is an immersion by Proposition ~\ref{wideoverlap}. Suppose that $\i:Y\to Y'$ is not an immersion. Then there exists a band $\BB_{Y'} \subset Y'$ which has two preimages $\BB^1_{Y},\BB^2_{Y}\subset Y$ with the property that $\BB^1_{Y}$ and $\BB^2_{Y}$ have an wide overlap. Now, fold $\BB^1_{Y}$ and $\BB^2_{Y}$ according to the overlap. Let the resulting band complex be $X^*$ and $\i^*: X^*\to X'$ be the induced morphism. In particular, $\i^*$ maps $ \BB_{Y^*}$ into $ \BB_{Y'}$ where $\BB_{Y^*}=\BB^1_{Y}\cup \BB^2_{Y}/\sim_{fold}$ is the new band in $Y^*$. Apply Proposition ~\ref{wideoverlap} again if necessary to make sure that $\i^*:\Gamma_{X^*}\to \Gamma_{X'}$ is an immersion. The number of generalized bands in $Y^*$ is less than the number in $Y$ due to the fold. Therefore, after finitely many steps of folding, $\i^*$ will be an immersion.
\end{proof}

Proposition~\ref{maimmersion} allows us to make the following assumption to any given pair of band complexes $(X\overset{\i}{\to}X')$.
 
\begin{itemize}
\item $(A4)$ The restriction of $\i$ to the level of unions of bands $\i:Y\to Y'$ is an immersion.
\end{itemize}

\vspace{1pc}
Let $Y'_0$ be a minimal component of $Y'$. Its $\i$-preimage may consists several components in $Y$. In each step of the relative Rips machine, we will preform moves on $Y'_0$ first, then apply induced moves to its preimages and modify $\i$ correspondingly. To simplify our notation, we will now focus on a special case satisfying $(A5)$, then come back to the general case in Section~\ref{generalcase}.

\begin{itemize}
   \item $(A5)$ For a fixed minimal component $Y'_0$ in $Y'$, its $\i$-preimage is either empty or a unique component $Y_0$.
\end{itemize}

\begin{remark}\label{realgraphnbhd}
   Here we note that to describe a morphism $\i: Y\to Y'$ between two unions of bands, one only need to specify a well-defined map between some neighborhoods of their underlying real graphs. Since morphisms map every generalized band $\BB_Y$ in $Y$ homeomorphically to a sub-generalized band in $Y'$. A map between neighborhoods of $\Gamma_{Y}$ and $\Gamma_{Y'}$ is \textit{well defined} if for any $\BB_Y\in Y$, neighborhoods of its bases map to dual positions that uniquely determine a band $\BB_{Y'}$ in $Y'$. Such map induces a unique morphism from $Y$ to $Y'$ (up to homotopies within bands) by sending $\BB_Y$ to $\BB_{Y'}$. 
    Moreover, for a given morphism $\i:Y\to Y'$, to show $\iota$ is an immersion (a submersion), one only need to check local injectivity (surjectivity) near the real graphs. We will use this fact throughout the rest sections. 
\end{remark}

\vspace{1pc}

Let $(X\overset{\i}{\to}X')$ be a pair of band complexes with their underlying unions of bands $(Y\overset{\i}{\to}Y')$. 
In Section~\ref{process}, often several moves described in Section~\ref{background} will be applied as units. 
We list these combinations of moves here for a better reference. 
 \vspace{1pc}

 (M9) \textbf{Subdivide $X$ at a point}. Let $q\in \Gamma$ be a point on the real graph of $X$. We obtain a new band complex $X^*$ by apply (M2) \textit{subdivide a band} to bands with the property that one of whose bases contains $q$ as an interior point in $Y$. $X^*$ is well-defined since there are only finitely many such bands in $Y$ (See Figure~\ref{F3}). 
 $(X^*\overset{\i^*}{\to}X')$ is a new pair of band complexes with immersion $\i^*$ defined as follows. On the level of union of bands, $\i^*: Y^*\to Y'$ is the composition of inclusion $Y^*\hookrightarrow Y$ and $\i: Y\to Y'$. On the level of attaching cells, $\i^*$ maps subdivision cells newly created from splitting to their corresponding vertical fibers in $X'$ and remains the same on all the other attaching cells.

\begin{figure}[h!]
\centerline{\includegraphics[width=1\textwidth]{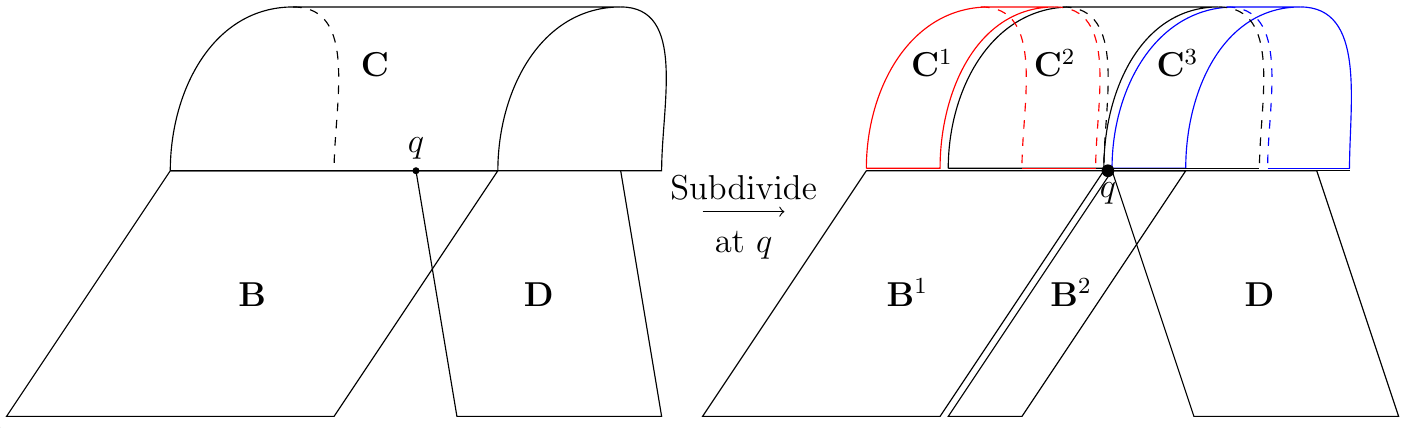}}
\caption{\footnotesize There are three generalized bands $\BB,\mathbf{C}, \mathbf{D}$ containing point $q$. In particular, one base of $\BB$ and both bases of $\CC$ contain $q$ as an interior point. Therefore (M2) is applied three times during the move of subdividing $X$ at $q$.} \label{F3}
\end{figure}

 (M10) \textbf{Subdivide $X'$ at a point}. Let $p\in \Gamma'$ be a point on the real graph of $X'$ and $q_1, \dots, q_n\in \Gamma$ be its $\i$-preimages in $X$. Using (M9), we may subdivide $X'$ at $p$ to obtain ${X'}^*$. To form  a new pair $(X^*\overset{\i^*}{\to}{X'}^*)$, we can then obtain $X^*$ by subdividing $X$ at $q_1, \dots, q_n$ and define $\i^*: X^*\to {X'}^*$ to be the induced map from $\i$. In particular, $\i^*$ maps newly created subdivision cells in $X^*$ to the corresponding newly created subdivision cells in ${X'}^*$.

  \vspace{1pc}    
(M11) \textbf{Duplicate a segment of $\Gamma$}. Let $c$ be a segment of $\Gamma$. First, using (M6) we add an extra segment $c'$ of length $l(c)$ to $\Gamma$. Then applying (M7) we attach to $Y$ a new band $\CC=c\times I$ via measure-preserving maps $c\times \{0\}\to c$ and $c\times \{1\}\to c'$. It is clear that there are two copies of $c$ in the new $Y$ and the new $Y$ resolves the same tree as the original one. 
      In particular, a band $\BB\subset Y$ with one of its bases $b$ contained in $c$ can now be attached to the new $Y$ either along $c$ or $c'$ (slide $\BB$ across $\CC$).   
      
      We will use (M11) to solve the problem of keeping $\i$ as an immersion in the situation that there exist overlapping bases in $Y$ mapping to disjoint bases after the application of certain moves on $Y'$. See Section~\ref{processII} for details.

      

\begin{notation}
  For the rest of this paper unless otherwise stated, we will use the following notation:
  
  \begin{itemize}
     \item $(X\overset{\i}{\to}X')$ is a pair of band complexes satisfying $(A1)-(A5)$;
     \item $(Y\overset{\i}{\to}Y')$ is its pair of underlying unions of bands;
     
     \item $(Y_0\overset{\i_0}{\to}Y'_0)$ is a pair of components where $Y_0\subset Y$, $Y'_0\subset Y$ and $\i_0=\i|_{Y_0}$. 
     
     \item Denote $\i(Y)\subset Y'$ by $\widehat{Y}$ and $\i(Y_0)\subset Y'_0$ by $\widehat{Y}_0$;
     \item $\BB_Y\subset Y$ is a generalized band with bases $b_Y$ and $dual(b_Y)$ (Similarly for other generalized bands, e.g. $c_Y$ and $dual(c_Y)$ for $\mathbf{C}_Y\subset Y$, $b_{Y'}$ and $dual(b_{Y'})$ for $\BB_{Y'}\subset Y'$).    
  \end{itemize}   
\end{notation}

\section{Relative Rips machine}\label{process}
We will describe the relative Rips machine in this section. This machine consists three processes (sequences of moves): Process~I, Process~II and Process~III. Fix a pair of components $(Y_0\overset{\i_0}{\to}Y'_0)$ in 
$(X\overset{\i}{\to}X')$. In each step of Process I and Process II, we first apply a move to $Y'_0$. It is the same move one would apply to $Y'_0$ in the Rips machine. Then we apply moves to $Y_0$ and modify $\i_0$ correspondingly so that $\i_0$ remains an immersion. Successive application of Process I and Process II will convert $Y'_0$ into its \textit{standard form} (Definition~\ref{standardform}), whereas $Y_0$ may not be in its standard form yet. In particular, there may exist weight $1$ bands in $Y_0$ mapping to weight $0$ or weight $\frac{1}{2}$ bands in $Y'_0$. Such bands in $Y_0$ are called pre-weight $0$ bands or pre-weight $\frac{1}{2}$ bands. Process III is then needed to deal with such bands (Section~\ref{processIII}). After finitely many steps in Process III, one will obtain a new pair of components, denoted still by $(Y_0\overset{\i_0}{\to}Y'_0)$, with the property $(*_1)$ that $Y'_0$ is in a standard form and bands in $Y_0$ have the same weights as their images in $Y'_0$. The machine then go back to Process I with $(Y_0\overset{\i_0}{\to}Y'_0)$. Property $(*_1)$ ensures that Process III will not appear again. In Section~\ref{machineoutput}, as machine output, we will show that one is able to tell the type of $Y_0$ as the machine successively applied and the machine will eventually also convert $Y_0$ into a standard form if $Y_0$ is of the surface or thin type.
Once we finish analyzing $(Y_0\overset{\i_0}{\to}Y'_0)$, we continue by choosing another pair. 

\subsection{Process I} \label{processI}

Let $(Y_0\overset{\i_0}{\to}Y'_0)$ be a pair of components. $(X^*\overset{\i^*}{\to}{X'}^*)$ is defined to be a pair of band complexes obtained from $(X\overset{\i}{\to}X')$ by the following operation. Find, if possible, a maximal free subarc $J_{Y'}$ of a base $b_{Y'}$ in $Y'_0$. If such a $J_{Y'}$ does not exist, define $(X^*\overset{\i^*}{\to}{X'}^*)=(X\overset{\i}{\to}{X'})$ and go on to Process II. Else in $Y'_0$, use $(M5)$ to collapse from $J_{Y'}$ to get $Y'_1$. If there are several $J_{Y'}$'s to choose from, abide the same rule as the Rips machine (see a review in Section~\ref{thinreview}).

Now we need to make corresponding changes for $Y_0$ and $\i_0$ to obtain $Y_1$ and $\i_1$. There are two cases depending on whether the collapsed region in $Y'_0$, i.e. the union of $J_{Y'}$ and the interior of $\JJ_{Y'}=J_{Y'}\times I_n$ ($J_{Y'}$ is identified with $J_{Y'}\times \{0\}$), intersects $\widehat{Y}_0$. If the collapsed region in $Y'_0$ does not intersect $\widehat{Y}_0$, let $Y_1=Y_0$ and $\iota_1=\iota_0$.

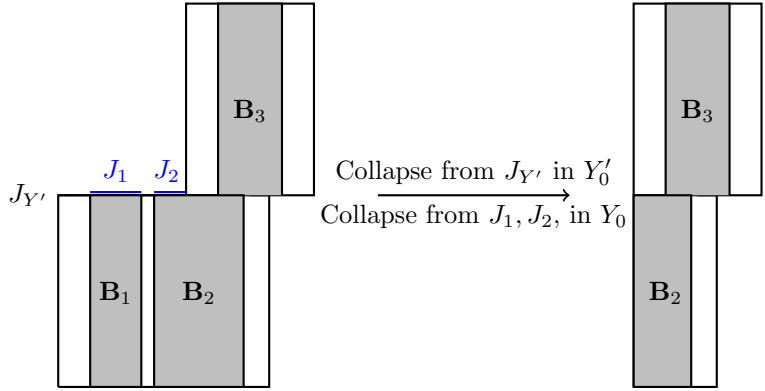
\begin{figure}[h]
   \centering
\begin{tikzpicture}[thick, scale=0.85]
\footnotesize

\draw (0,-3)--(0,0)--(3.3,0)--(3.3,-3)--(0,-3);
\draw (2,0)--(2,3)--(4, 3)--(4,0)--(3,0);
\draw [fill=lightgray] (2.5,0)--(2.5,3)--(3.5,3)--(3.5,0);
\node [above] at  (3,1) {$\BB_3$};
\draw[thick] (0,0)--(2,0);
\node [left] at (0,0) {$J_{Y'}$};

\draw [fill=lightgray] (0.5,0)--(0.5,-3)--(1.3,-3)--(1.3,0)--(0.5,0);
\draw [blue][thick] (0.5,0.05)--(1.3,0.05);

\node [blue][above] at (0.9,0.05) {$J_1$};
\node [below] at (0.9, -1.2) {$\BB_1$};

\draw [fill=lightgray] (1.5,0)--(1.5,-3)--(2.9,-3)--(2.9,0)--(1.5,0);
\draw[blue][thick] (1.5,0.05)--(2,0.05);
\node [blue][above] at (1.7,0.05) {$J_2$};
\node [below] at (2.2, -1.2) {$\BB_2$};

\draw [->] (5,0)--(8,0);
\node [above]  at (6.5,0) {Collapse from $J_{Y'}$ in $Y'_0$};
\node [below]  at (6.5,0) {Collapse from $J_1,J_2,$ in $Y_0$};

\draw (9,-3)--(9,0)--(10.3,0)--(10.3,-3)--(9,-3);
\draw (9,0)--(9,3)--(11, 3)--(11,0)--(10,0);
\draw [fill=lightgray] (9.5,0)--(9.5,3)--(10.5,3)--(10.5,0);
\node [above] at  (10,1) {$\BB_3$};
\draw [fill=lightgray] (9,0)--(9,-3)--(9.9,-3)--(9.9,0)--(9,0);
\node [below] at (9.5, -1.2) {$\BB_2$};

\end{tikzpicture}
\caption{\footnotesize{The figure is showing the case that $Y_0$ is identified with $\widehat{Y}_0$ near $J_{Y'}$.
Blank rectangles are bands in $Y'_0$ and gray parts are bands in $Y_0$.}}\label{F1}
\end{figure}

Otherwise let $\{\JJ_i\}_{i=1,\dots, n}$ be the set of preimages of $\JJ_{Y'}\cap \widehat{Y}_0$ in $Y_0$, and let $J_i$ be the base of $\JJ_i$ whose $\i_0$-image is contained in $J_{Y'}$.  Since $\iota_0$ is an immersion, $J_i$'s are also free subarcs. 
$Y_1$ is produced by collapsing from $\{J_i\}_i$ in $Y_0$. Then restricting $\iota_0$ to $Y_1$, we get $\iota_1:Y_1\to Y'_1$ (See Figure~\ref{F1}). In particular, $\iota_1$ maps subdivision annuli created by collapsing from $\{J_i\}_i$ to subdivision annuli created by collapsing from $J_{Y'}$.

For the case where $\JJ_{Y'}$ is contained in a weight $\frac{1}{2}$ band, each $\JJ_i$ is contained in either a weight $\frac{1}{2}$ or weight $1$ band. The above process is well defined. In more detail, suppose that the band containing $\JJ_{Y'}$ is $\BB_{Y'}$ of weight $\frac{1}{2}$ and the band containing $\JJ_i$ is $\BB_Y^i$.  Let $w$ be the midpoint of the base $b_{Y'}$ of $\BB_{Y'}$. Recall that to collapse from $J_{Y'}$, we first replace $\BB_{Y'}$ by bands of weight $0,\frac{1}{2}$, and $1$ (In the degenerate case where $w$ is one end point of $J_{Y'}$, there is no weight $\frac{1}{2}$) so that $J_{Y'}$ is a free arc contained in a weight $1$ base, see \cite[section 6]{bfstableactions}.  
In each $\BB_Y^i$, we can do the corresponding replacement such that  
each $\JJ_i$ is contained in a weight $1$ band. We may then collapse from $J_{Y'}$ in $Y'_0$ and collapse from $J_i$'s in $Y_0$ as described above.

In either case, let the resulting pair of band complexes containing $(Y_1\overset{\i_1}{\to}Y'_1)$ be $(X^*\overset{\i^*}{\to}{X'}^*)$. $\i^*$ is the same as $\i_1$ when restricted to $Y_1$ and is the same as $\i$ on other components. Thus $\i^*$ remains an immersion. We then continue by applying Process I to $(X^*\overset{\i^*}{\to}{X'}^*)$.
 
Moreover since restricting to $X'$, all we have done are exactly the same as Process I of the Rips machine, so the following holds (See a proof in the first item of Proposition~\ref{thinprop}). 
$$Complexity ({X'}^*) \leq Complexity(X').$$

\subsection{Process II}\label{processII}
Follow the same notation as above, we now describe Process II which again will produce from $(X\overset{\i}{\to}X')$  another pair of band complexes $(X^*\overset{\i^*}{\to}{X'}^*)$. Process II will only be applied after Process I. So, in this section, we also assume that $Y'_0$ has no free subarc ($Y_0$ may contain some free subarcs), i.e. for each point $z\in \Gamma_{Y'_0}$, we have

$(*_2)$: \textit{the sum of weights of the bases containing $z$ is at least 2. }

\vspace{1pc}
In particular, if the Rips machine constructed in \cite{bfstableactions} takes $Y'_0$ as an input, its Process II will be applied. Moves described below for $Y'_0$ (to obtain $Y'_1$) is in fact a review for Process II of the Rips machine.

\vspace{1pc}

We may orient $\Gamma_{Y'_0}$ and order the components of $\Gamma_{Y'_0}$. This induces a linear order on $\Gamma_{Y'_0}$. Let $K$ be the first component and $z$ be the initial point of $K$. Let $b_{Y'}$ be the longest base of positive weight containing $z$, chosen to have weight $1$ if possible. Further let $\BB_{Y'}$ be the generalized band with $b_{Y'}$ one of its bases. The union of bands $Y'_1$ is the result of the composition of the following two operations.

\textit{Operation 1} (Slide) $Y'_0\to Y'^*_0$: Slide from $b_{Y'}$ all these positive weight bases contained in $b_{Y'}$ (except $b_{Y'}$ and $\textit{dual}(b_{Y'})$) whose midpoint is moved away from $z$ as a result of the slide. 

\textit{Operation 2} (Collapse) $Y'^*_0\to Y'_1$: Collapse from the maximal free initial segment of $b_{Y'}$.

\vspace{1pc}
Now, we only need to define $Y_1$ and $\iota_1: Y_1\to Y'_1$. Once $Y_1$ and $\iota_1$ have been properly defined, since $$Complexity (Y'_1)\leq Complexity(Y'_0)$$ and if they are equal, $(*_2)$ holds for $Y'_1$ (\cite[Proposition 7.5]{bfstableactions}), we are in position to apply Process II again.  Process II is then successively applied unless the complexity decreases at some stage. We then say Process II sequence ends and go back to Process I.  

In the following, we describe how to obtain $Y_1$ and $\iota_1$ by cases depending on whether $\widehat{Y}_0$ intersects $\BB_{Y'}$.

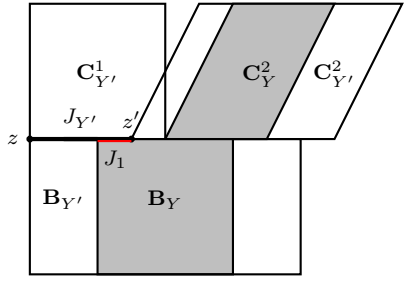
\begin{figure}[h]
\centering

\begin{tikzpicture}[thick, scale=0.9]
\draw (0,0)--(0,2)--(2,2)--(2,0);
\draw(1.5,0)--(2.5,2)--(5.5,2)--(4.5,0)--(1.5,0);

\draw[fill=lightgray](2,0)--(3,2)--(4.5,2)--(3.5,0)--(1.75,0);



\draw [fill=lightgray](1,0)--(3,0)--(3,-2)--(1,-2)--(1,0);
\draw (0,0)--(4,0)--(4,-2)--(0,-2)--(0,0);

\coordinate (z) at (0,0);
\filldraw(z) circle (1pt) node[left] {\tiny{$z$}};

\coordinate (z') at (1.5,0);
\filldraw(z') circle (1pt) node[above] {\tiny{$z'$}};

\draw[ultra thick] (0.5,0)--(1,0);
\node [above] at (0.75,0) {\tiny{$J_{Y'}$}};

\draw [ultra thick] (0,0)--(1.5,0);

\draw[red][thick](1,-0.03)--(1.5,-0.03);
\node[left] at (1.6,-0.3) {\tiny{$J_1$}};

\node[above] at (1,0.6) {\tiny{$\CC^1_{Y'}$}};
\node[above] at (4.5,0.6) {\tiny{$\CC^2_{Y'}$}};
\node[above] at (3.4,0.6) {\tiny{$\CC^2_Y$}};
\node[above] at (0.5,-1.2){\tiny{$\BB_{Y'}$}};
\node[above] at (2,-1.2){\tiny{$\BB_Y$}};

\end{tikzpicture}  

\caption{\footnotesize{The figure is showing a case where $Y_0$ is contained in $Y'_0$ near $\BB_{Y'}$. Bands in $Y'_0$ are blank rectangles and bands in $Y_0$ are shown in gray. Suppose that in $Y'_0$, $\BB_{Y'}$ is the carrier, then $\CC^1_{Y'}$ will be carried and $\CC^2_{Y'}$ will not. $J_1$ is shown in red.}} \label{F2}
\end{figure}

Let $J_{Y'}$ be the maximal free initial segment of $b_{Y'}$ produced in operation 1 
and let the sub-generalized band collapsed in operation 2 be $\JJ_{Y'}$. We may further assume that:

   $(*_3)$: \textit{In $Y'_0$, $\widehat{Y}_0$ intersects some generalized band that is carried in operation 1.}

Otherwise, if further $\widehat{Y}_0$ also has no intersection with the interior of $\JJ_{Y'}$, we are done by letting $Y_1=Y_0$ and $\iota_1=\iota_0$; If $\widehat{Y}_0$ only intersects $\JJ_{Y'}$ (See Figure~\ref{F2}), let the set of preimages in $Y_0$ of these intersections be $\JJ_1,\dots, \JJ_k $. Let $J_i$ be a base of $\JJ_i$ as in Process I. $Y_1$ is obtained by collapsing from all $J_i$'s in $Y_0$ and $\iota_1$ is the restriction of $\iota_0$ on $Y_1$.

Under assumption $(*_3)$, there are now two cases.
\newline

\textbf{Case 1}: $\widehat{Y}_0$ has no intersection with the carrier $\BB_{Y'}$.

Let $c_Y$ be a base of a generalized band $\mathbf{C}_Y$ in $Y_0$. If $\iota_0(c_Y)$ is contained in a base that is carried by $\BB_{Y'}$, we also say $c_Y$ is \textbf{carried}. Otherwise, we say $c_Y$ is \textit{non-carried}. 
According to Remark~\ref{realgraphnbhd}, to define $Y_1$ and $\i_1$ we only need to work on some neighborhood of $\Gamma_{Y_0}$ in $Y_0$. The definition of $Y_1$ and $\i_1$ will be made block by block depending on whether bases contained in that block are carried.

Let $I$ be a block of $\Gamma_{Y_0}$. If every base contained in $I$ is non-carried, 
 $Y_1$ and $\i_1$ are defined to be the same as $Y_0$ and $\iota_0$ near this block. 
 
 If every base contained in $I$ is carried, then define $Y_1$ to be the same as $Y_0$ near this block. $\i_1$ is defined by mapping every base $c_Y \subset I$ to the dual position of $\iota_0(c_Y)$ in $dual(b_{Y'})$. To be more precise, let $\CC_{Y'}$ be the band in $Y'_0$ containing the $\i_0$-image of $\CC_Y$, where $\CC_Y\subset Y_0$ is the band with one of its bases $c_Y$. During the operation $Y'_0\to Y'_1$, $\CC_{Y'}$ is slid across $\BB_{Y'}$ to a new position in $Y'_1$. $\i_1$ is defined to map $\CC_Y$ into $\CC_{Y'}$'s new position in $Y'_1$.  Otherwise, 
\newline

$(*_4)$: \textit{Within a block $I$, some bases are carried and some bases are not.}

Near blocks satisfying $(*_4)$, $Y_1$ is defined as a composition of the following two operations.
\newline

\textit{Operation $1'$} (Slide) $Y_0\to Y^*_0$: 
Let the union of carried bases within $I$ be $\omega = \omega_1\cup \dots \cup \omega_k$, where each $\omega_i$ is a connected component of $\omega$. Near $I$,  we define $Y^*_0$ by modifying $Y_0$ as below. Use (M11) to duplicate each $\omega_i$ by first adding a new copy of $\omega_i$ to $\Gamma_{Y_0}$ and then attaching a new band $\Omega_i=\omega_i\times [0,1]$ along $\omega_i$ and its new copy. Denote the other base of $\Omega_i$ (the new copy of $\omega_i$) by $dual(\omega_i)$ (See Figure ~\ref{F5}). In $I$, every carried base $c_Y$ is contained in some $\omega_i$. Slide each band $\CC_Y$ that corresponds to some $c_Y$ across $\Omega_i$. 
Do this modification on every block satisfying $(*_4)$ and call the resulting union of bands $Y^*_0$. 

By the above construction, $Y^*_0$ is the same as $Y_0$ away from a neighborhood of $\Omega_i$'s. To define $\iota^*_0: Y^*_0\to {Y'_0}^*$, we only need to describe the map near $\Omega_i$'s.  
In $\Gamma_{Y^*_0}$, firstly, consider a block that contains some $\omega_i$. All bases other than $\omega_i$'s in this block are non-carried bases came from $Y_0$. So define $\i^*_0=\i_0$ near these bases and further define $\i^*_0(\omega_i)=\iota_0(\omega_i)\subset b_{Y'}$. Secondly, consider a block of $\Gamma_{Y^*_0}$ containing some $dual(\omega_i)$.  Each base, $c_Y$, other than $dual (\omega_i) $ in this block is a carried base came from $Y_0$. Thus define $\i^*_0$ to map $c_Y$ to the new position of $\iota_0 (c_Y)$ in $Y'^*$ and define $\iota^*_0(dual(\omega_i))=dual(\iota_0(\omega_i))$ in $dual(b_{Y'})$. Finally, let the unique subband of $\BB_{Y'}$ determined by the image of $\omega_i$ and the image of $dual(\omega_i)$ be $\widehat{\Omega}_i$. Define $\iota^*_0(\Omega_i)=\widehat{\Omega}_i$. 

Now let's check that $\i^*_0$ is an immersion. Near the blocks of $\Gamma_{Y^*_0}$ containing some $\omega_i$'s, $\i^*_0$ is the restriction of $\i_0$ to all bands other than $\Omega_i$'s. The fact that $\i_0$ is an immersion and $\widehat{Y}_0$ does not intersect $\BB_{Y'}$ implies that $\i^*_0$ is injective near these blocks. Near the blocks of $\Gamma_{Y^*_0}$ containing some $dual(\omega_i)$'s, $\i^*_0$ is injective for the same reason. 
\newline


\textit {Operation $2'$} (Collapse) $Y^*_0\to Y_1$: Let $\widehat{Y}^*_0\subset {Y'}^*_0$ be the image of $Y^*_0$ under $\iota^*_0$. By the construction of operation $1'$, $\widehat{Y}^*_0$ may intersect $\JJ_{Y'}$ (Images of $\Omega_i$ may intersect $\JJ_{Y'}$). The intersection part will collapse in this operation. Exactly the same as in Process I, let the set of preimages of these intersections be $\JJ_i$'s.  
$Y_1$ is obtained by collapsing from $J_i$'s in $Y^*_0$.  $\iota_1$ is then defined to be the restriction of $\iota^*_0$ to $Y_1$. Hence, it is an immersion since $\iota^*_0$ is.
\newline

\begin{figure}[h]
\centering
\begin{tikzpicture}[scale=0.8]
\footnotesize

\node [left] at (-0.5, 0) {$Y'$:};
\draw [fill=gray](0,-2)--(0,0)--(1.5,0)--(1.5,-2)--(0,-2);

\draw (0,0)to [out=90, in=180] (1.5,2)--(3.5,2)to [out=180,in=90](2,0)--(0,0); 
\draw (1.5,2)--(3.5,2) to [out=0, in=90] (5,0)--(3,0); 
\draw [dashed](1.5,2)to [out=0, in=115] (2.5,1.4);
\draw  (3,0) to [out=90, in=115] (2.6,1.3);
\draw [fill=gray](2,-2)--(1,0)--(2.5,0)--(3.5,-2)--(2,-2);
\draw [dashed] (1.5,0)--(1.5,-1);
\node [above] at (1.2,0.8) {$\BB_{Y'}$};
\node [above] at (0.75,-1.2) {$\CC_{Y'}$};
\node [above] at (2.25, -1.2) {$\mathbf{A}_{Y'}$};

\draw [->] (5.5,0)--(6.5,0);


\draw [fill=gray](10,-2)--(10,0)--(11.5,0)--(11.5,-2)--(10,-2);
\draw [fill=gray](9,-2)--(8,0)--(9.5,0)--(10.5,-2)--(9,-2);

\draw  (10,0) to [out=90, in=115] (9.6,1.3);

\draw [dashed][fill=lightgray](10,2) to [out=0, in=90] (11.5,0)--(10,0)to [out=90, in=115] (9.6,1.3) to[out=115,in=0](8.5,2); 
\draw [dashed][fill=gray] (7,0)to [out=90, in=180] (8.5,2)--(10,2) to [out=180, in=90] (8.5,0);
\draw (7,0)to [out=90, in=180] (8.5,2)--(10.5,2)to [out=180,in=90](9,0)--(7,0); 
\draw (8.5,2)--(10.5,2) to [out=0, in=90] (12,0)--(10,0)to [out=90, in=115] (9.6,1.3) ;
\node [above] at (8.2,0.8) {$\BB_{Y'}$};
\node [above] at  (10.5, 0.8) {$\widehat{\Omega}$};
\node [above] at (10.75,-1.2) {$\CC_{Y'}$};
\node [above] at (9.25,-1.2) {$\mathbf{A}_{Y'}$};
\draw[dashed] (10,-1)--(10,-2);

\end{tikzpicture}

\begin{tikzpicture}[scale=0.8]
\footnotesize
\node [left] at (-0.5, 0) {$Y$:};
\draw [fill=gray](0,-2)--(0,0)--(1.5,0)--(1.5,-2)--(0,-2);
\draw [fill=gray](2,-2)--(1,0)--(2.5,0)--(3.5,-2)--(2,-2);
\draw [dashed] (1.5,0)--(1.5,-1);
\draw [ultra thick] (0,0)--(1.5,0);
\node [above] at (0.75,0) {$\omega$};
\node [above] at (0.75,-1.2) {$\CC_Y$};
\node [above] at (2.25, -1.2) {$\mathbf{A}_Y$};

\draw [->] (1.5,0.5)--(1.5,1.5);
\draw [->] (2.6,0.5)--(3.8,0.5);

\draw (4,-2)--(4,0)--(5.5,0)--(5.5,-2)--(4,-2);
\draw (6,-2)--(5,0)--(6.5,0)--(7.5,-2)--(6,-2);
\draw (4,0)--(4,2)--(5.5,2)--(5.5,0)--(4,0);
\draw [ultra thick] (4,0)--(5.5,0);
\draw [ultra thick] (4,2)--(5.5,2);
\node [above] at (4.75,0) {$\omega$};
\node [above] at (4.75,2) {$dual(\omega)$};
\node [above] at (4.75,-1.2) {$\CC_Y$};
\node [above] at (6.25, -1.2) {$\mathbf{A}_Y$};
\node [above] at (4.75, 0.8) {$\Omega$};

\draw [->] (6,0.5)--(7.2,0.5);

\draw [fill=gray](11,-2)--(11,0)--(12.5,0)--(12.5,-2)--(11,-2);
\draw [fill=gray](10,-2)--(9,0)--(10.5,0)--(11.5,-2)--(10,-2);

\draw [fill=lightgray](11,2) to [out=0, in=90] (12.5,0)--(11,0) to [out=90, in =0] (9.5,2); 
\draw [fill=gray](8,0)to [out=90, in=180] (9.5,2)--(11,2) to [out=180, in=90] (9.5,0)--(8,0);
\draw [ultra thick] (9.5,0)--(8,0);
\node [above] at (8.75,0) {$\omega$};
\draw [ultra thick] (12.5,0)--(11,0);
\draw [dashed] (11,-1)--(11,-2);
\node [above] at (11.75,-0.1) {$dual(\omega)$};
\node [above] at  (11.5, 0.8) {$\Omega$};
\node [above] at (11.75,-1.2) {$\CC_Y$};
\node [above] at (10.25,-1.2) {$\mathbf{A}_Y$};
\draw [->] (9.25,2.2)--(9.25,2.7);

\end{tikzpicture}
\caption{\footnotesize The shaded part is $\widehat{Y}$ (within ${Y'}$).} \label{F5}   
\end{figure}
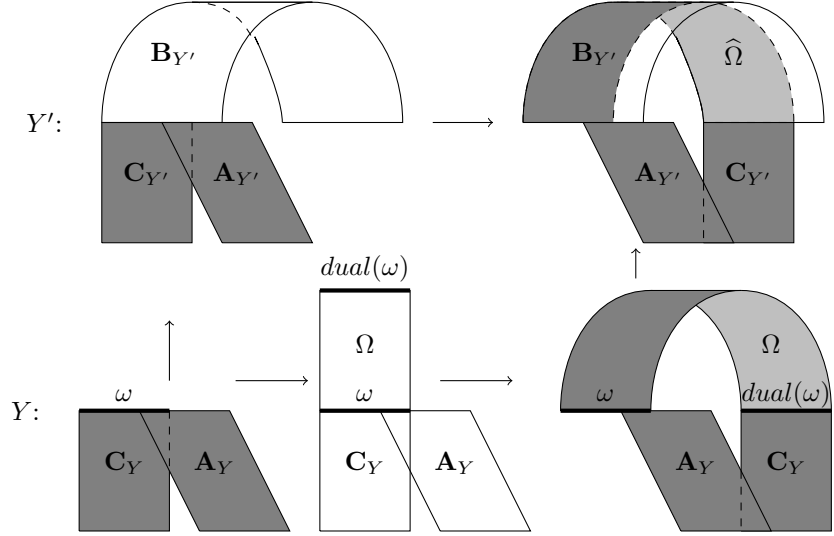

\textbf{Case 2}: $\widehat{Y}_0$ intersects $\BB_{Y'}$.

\textit{Operation $0'$} (Subdivide) Let $\BB^1_{Y}, \dots, \BB^k_{Y}$ be generalized bands of $Y_0$, whose images under $\iota_0$ are contained in $\BB_{Y'}$.
For each $\BB^i_{Y}$, denote its base that maps into $b_{Y'}$ by $b^i_{Y_0}$. Further let the end points of $b^i_{Y}$ be $z_i$ and $z'_i$ and the block of $\Gamma_{Y_0}$ containing $b^i_{Y}$ be $K_i$. If $z_i$ or $z'_i$ is an interior point of $K_i$, subdivide $Y_0$ at that point using (M9). Denote the resulting union of bands still by $Y_0$. Each band in the new $Y_0$ is a subband of an original band in $Y_0$. So the original $\i_0$ induces a map from the new $Y_0$ to $Y'_0$. The induced map, denoted still by $\i_0$, remains an immersion since every block in the new $Y_0$ is contained in an original block of $Y_0$. 

As a result of operation $0'$, we may assume the following property for each block $K$ of $Y_0$.

    $(*_5)$ \textit{For the set of generalized bands $\{\mathbf{C}_{Y}|c_{Y}\subset K\}$, either none of $\{\i_0(\CC_{Y})\}$ intersects the interior of $\BB_{Y'}$ or exactly one of $\{\i_0(\mathbf{C}_{Y})\}$ intersects the interior of $\BB_{Y'}$. Further in the latter type, the band whose $\i_0$-image intersects $\BB_{Y'}$ is one of $\{\BB^1_{Y}, \dots, \BB^k_{Y}\}$}. 

\vspace{1pc}
\textit{Operation $1'$} (Slide) $Y_0\to Y^*_0$: In $Y_0$, for blocks of the former type in $(*_5)$, we proceed exactly as in case 1 to get $Y^*_0$ and $\i^*_0$. Otherwise, for a block $K$, let $\BB_Y$ (one of $\{\BB^i_Y\}$) be the generalized band that map into $\BB_{Y'}$ with one of its bases $b_Y=K$ mapping into $b_{Y'}$. Further let $\CC_Y\subset Y_0$ be a generalized band with one of its bases $c_Y$ contained in $b_Y$, and let the band in $Y'$ containing its $\i_0$-image be $\CC_{Y'}$. Then $Y^*_0$ is obtained by sliding every carried $\CC_Y$ across $\BB_Y$. $\i^*_0$ is defined by mapping $\CC_Y$ (possibly in a new position of $Y^*_0$) into $\CC_{Y'}$ (possibly in a new position at ${Y'_0}^*$). It is well-defined since $\CC_Y$ is carried if and only if $\CC_{Y'}$ is carried. Thus $\i^*_0$ is also an immersion.

\textit {Operation $2'$} (Collapse) $Y^*_0\to Y_1$: Define $J_i$'s as in case 1 and collapse from $J_i$'s to get $Y_1$. In particular, for a given block $K$ in the latter type of $(*_5)$, if its image under $\iota_0$ is contained in $J_{Y'}$, all bases other than $b_Y$ in this block must be carried. Therefore such $K$'s become free arcs in $Y^*_0$, and hence completely collapse. Restricting $\iota^*_0$ to $Y_1$, we get $\i_1$ which is indeed an immersion since $\i^*_0$ is. 



\vspace{1pc}
\begin{prop} \label{mimicgraph}
   Let $(Y_0\overset{\i_0}{\to}Y'_0)$ be a pair of components where $Y'_0$ is either of surface or thin type, and $(Y_0\overset{\i_0}{\to}Y'_0), (Y_1\overset{\i_1}{\to}Y'_1),\dots$ be a sequence of pairs of components formed by Process I and Process II. 
There exists an integer $N$ such that for all $n\geq N$, $\widehat{Y}_n\cap \BB_{Y'_n}$ is a unique sub-generalized band where $\widehat{Y}_n$ is the image of $Y_n$ in $Y'_n$ and $\BB_{Y'_n}\subset Y'_n$ is a positive weighted generalized band. Further let $\BB^1_{Y_n}, \dots, \BB^{k_n}_{Y_n}$ be generalized bands in $Y_n$ that map into $\BB_{Y'_n}$,  then $\cap_{j=1}^{k_n} \widehat{\BB}^j_{Y_n}$ contains a vertical fiber that is in the limit set $\Omega_{Y'_n}$.  
\end{prop}
\begin{proof}
By construction, there exists an integer $N_0\geq 0$ such that for any pair $(Y_i\overset{\i_i}{\to}Y'_i)$ with $i\geq N_0$, $Y'_i$ is in a standard form. 
          
          First assume that there exists some $(Y_n\overset{\i_n}{\to}Y'_n)$ with $n\geq N_0$ with the desire property: 
          
          $(*_6)$ \textit{For any generalized band $\BB_{Y'_n}$ of a positive weight in $Y'_n$, $\widehat{Y}_n\cap \BB_{Y'_n}$ is a unique sub-generalized band and $\cap_{j=1}^{k_n} \widehat{\BB}^j_{Y_n} \cap \Omega_{Y'_n} \neq \emptyset$.}
          
          We now check that property $(*_6)$ also holds for $(Y_{n+1}\overset{\i_{n+1}}{\to}Y'_{n+1})$.
          
  Suppose that $(Y_{n+1}\overset{\i_{n+1}}{\to}Y'_{n+1})$ is obtained from $(Y_n\overset{\i_n}{\to}Y'_n)$ by Process~ I (i.e. $Y'_0$ is a thin component). Let $\JJ_{Y'_n}\subset Y'_n$ be the collapsed region. If $\widehat{Y}_n$ does not intersect the interior of $\JJ_{Y'_n}$, then $Y_{n+1}=Y_n$ and property $(*_6)$ is automatically preserved. Otherwise, we need to check property $(*_6)$ for \textit{images} (Notation~\ref{sequencenotation} ) of the band  $\BB_{Y'_n}$ containing $\JJ_{Y'_n}$. Let $\BB_{Y'_{n+1}}$ be an image of $\BB_{Y'_n}$ in $Y'_{n+1}$ ($\BB_{Y'_i}$ may have two images if the collapse is an $I_3$-collapse, see Section~\ref{thin} for more detail.) and $\{\BB_{Y_{n+1}}^j\}_j\subset Y_{n+1}$ be the set of $\i_{n+1}$-preimages of $\BB_{Y'_{n+1}}$. Note that each $\BB_{Y_{n+1}}^j$ is an image of some $\BB_{Y_n}^{j}$. One of the vertical boundaries of $\JJ_{Y'_n}$ is contained in $\BB_{Y'_{n+1}}$, denoted by $l_1$. According to the inductive assumption, $\cap_{j=1}^{k_n} \widehat{\BB}^j_{Y_n} \cap \Omega_{Y'_n}$ is not empty. Let $l_2$ be a vertical fiber in it. $\cap_{j=1}^{k_{n+1}} \widehat{\BB}^j_{Y_{n+1}}$ either contains $l_2$ (and we are done) or contains a neighborhood of $l_1$ in $\BB_{Y'_{n+1}}$. For the latter case, since limit set of a thin component $Y'_{n+1}$ is a prefect set, $(*_6)$ still holds.
         
          Suppose that $(Y_{n+1}\overset{\i_{n+1}}{\to}Y'_{n+1})$ is obtained from $(Y_n\overset{\i_n}{\to}Y'_n)$ by Process II (i.e. $Y'_0$ is a surface component). If $\BB_{Y'_{n}}$ is not the carrier for the step from $Y'_n$ to $ Y'_{n+1}$, $(*_6)$ holds for $\BB_{Y'_{n+1}}$ since there is a natural bijection between $\{\BB^j_{Y_{n+1}}\}_j$ and $\{\BB^j_{Y_{n}}\}_j$, and $\i_{n+1}(\BB^j_{Y_{n+1}})$ is identified with $\i_n(\BB^j_{Y_{n}})$ for all $j$. 
          Assume that $\BB_{Y'_{n}}$ is the carrier. 
          $Y'_n$ is a surface type implies that no (M11) duplicate a segment of the real graph is applied from $Y_n\to Y_{n+1}$. In particular, 
          preimages of $\BB_{Y'_{n+1}}$ in $Y_{n+1}$ are all produced from $\BB_{Y_n}^j$ by collapsing. Follow the same argument as above for Process I, $(*_6)$ also holds.
          
\vspace{1pc}        
Secondly, we show that there exists some $n>N_0$ such that $(*_6)$ holds for $(Y_n\overset{\i_n}{\to}Y'_n)$.
\newline

\textbf{Thin Case.}  In general, for a generalized band $\BB_{Y'_0}$ in $Y'_0$, $\widehat{Y}_0\cap \BB_{Y'_0}$ consists of finitely many sub-generalized bands. Given two generalized bands $\BB_{Y_0}^1, \BB_{Y_0}^2\subset Y_0$, let $\BB_{Y_i}^j$ be an image of $\BB_{Y_0}^j$ in $Y_i$ for $j=1,2$.
\newline

$(*_7)$ \textit{Claim that for $i$ sufficiently large, $\BB_{Y_i}^1$ and $\BB_{Y_i}^2$ map to the same band $\BB_{Y'_i}\subset Y'_i$ if and only if $\widehat{\BB}_{Y_i}^1\cap \widehat{\BB}_{Y_i}^2$ contains a vertical fiber in the limit set. }

 The `if' direction of the claim is clear. To prove the `only if' direction, suppose that $\BB_{Y_i}^1$ and $\BB_{Y_i}^2$ map to the same band and $\widehat{\BB}_{Y_i}^1\cap \widehat{\BB}_{Y_i}^2$ does not contain any vertical fiber in the limit set. Then there exists a vertical fiber $l'\subset \BB_{Y'_i}$, lying between the interior of $\widehat{\BB}_{Y_i}^1$ and the interior of $\widehat{\BB}_{Y_i}^2$, that is not in the limit set. 
There exists some $m>0$ such that $l'$ collapse from $X_{i+m}\to X_{i+m+1}$. Then then images of $\BB_{Y_i}^1$ and $\BB_{Y_i}^2$ in $Y_{i+m+1}$ map into different bands in $Y'_{i+m+1}$. So the claim $(*_7)$ holds.

We may pick $n$ sufficiently large such that $(*_7)$ holds for all choices of $\BB_{Y'_0}$'s and $\BB_{Y_0}$'s. Thus property $(*_6)$ follows. 
\newline



\textbf{Surface Case.} If $Y'_0$ is a surface component, then eventually only Process~II is applied. For $m>N_0$, $Y'_m$ is in a standard form and there is an infinite subsequence $m=m_1<m_2<\dots$ such that $Y'_{m_{i}}$ is a scaling down version of $Y'_{m_{i-1}}$. Thus all generalized bands in $Y'_m$ are getting thinner and their limits are the boundary leaves. Therefore, over all components of $\i_m(Y_m)\cap \BB_{Y'_m}$, only the one containing the limiting boundary (if there is any) survives till the end. Thus $\cap_{1}^{k_m}\BB_{Y_m}$ contains a neighborhood of the limiting boundary and so we are done. 
\end{proof}

\begin{definition}\label{apcover}
   Let $\i: Y\to Y'$ be a morphism between two unions of bands. $\i$ is \textbf{graph like} if each generalized band in $Y$ maps onto a generalized band in $Y'$. If $\i$ is a graph like immersion, we say $\i$ is a \textbf{partial covering map}. Further an immersion $\i$ is \textbf{an almost partial covering map} if every pair of components $(Y_0\overset{\i_0}{\to}Y'_0)\subset (Y\overset{\i}{\to}Y')$ has property $(*_6)$ described in Proposition~ \ref{mimicgraph}. 
   
  Thus Process I and Process II will eventually convert $Y'$ into a standard form and turn $\i_0: Y_0\to Y'_0$ into an almost partial covering map when $Y'_0$ is a surface or thin component. We say $(Y\overset{\i}{\to}Y')$ is eventually \textbf{stabilized} in this case.

\end{definition}

\subsection{Process III}\label{processIII}

Let $(Y_0\overset{\i_0}{\to}Y'_0)$ be a pair of components and $\BB_{Y'_0}\subset Y'_0$ be a generalized band. If $weight(\BB_{Y'_0})=1$, then all of its preimages in $Y_0$ also have the same weight. On the other hand, if $weight(\BB_{Y'_0})=0$ or $weight(\BB_{Y'_0})=\frac{1}{2}$, it is possible that instead of having the same weight as $weight(\BB_{Y'_0}$), some of its preimages in $Y_0$ have weight $1$ \footnote{Note that a preimage of a weight $0$ band (resp. $\frac{1}{2}$ band) can not be a weight $\frac{1}{2}$ band (resp. $0$ band).}. 
Such a weight $1$ preimage is called a \textbf{pre-weight $0$ band} or a \textbf{pre-weight $\frac{1}{2}$ band} (depending on the weight of $\BB_{Y'_0}$). Moreover, we say that $\i_0:Y_0\to Y'_0$ has a certain property \textbf{up to weight $0$ bands} if the restriction of $\i_0$ to the preimage of $\overline{Y'_0}$ in $Y_0$ has that property. 

One goal of the relative Rips machine is to convert both $Y'_0$ and $Y_0$ into some normal form simultaneously. Let $(Y_0\overset{\i_0}{\to}Y'_0), (Y_1\overset{\i_1}{\to}Y'_1),\dots$ be a sequence of pairs of components formed by Process I and Process II. As from Section~\ref{processII}, $(Y_n\overset{\i_n}{\to}Y'_n)$ stabilizes for sufficiently large $n$. We would like to conclude that $Y_n$ is also in a standard form by then. 
However, $Y_n$ may fail to be standard due to the existence of pre-weight $0$ or pre-weight $\frac{1}{2}$ bands  (see Example ~\ref{covernotstandard}). In Process III, we will ``get rid'' of pre-weight $0$ bands and pre-weight $\frac{1}{2}$ bands.  Intuitively, since two bases of a pre-weight $0$ (or $\frac{1}{2}$) band have the same image in $Y'_0$, we may then collapse pre-weight $0$ and pre-weight $\frac{1}{2}$ bands by viewing them as ultra thick real graphs. More detail are discussed below.


\begin{example}\label{covernotstandard}
Let $Y'_n$ be a minimal component in a standard form and $Y_n$ be a double cover of $Y'_n$. Suppose that $\BB_{Y'_n}\subset Y'_n$ is a weight $0$ band and its two preimages in $Y_n$ are $\BB_{Y_n}^1, \BB_{Y_n}^2$. Assume that $\BB_{Y_n}^1$ and $ \BB_{Y_n}^2$ are both pre-weight $0$ bands and they form an annulus in $Y_n$ as shown in Figure~\ref{F9}. $Y_n$ is then definitely not in a standard form since the complexity of $Y_n$ goes down if we slide $\BB_{Y_n}^1$ across $\BB_{Y_n}^2$.
\end{example}

  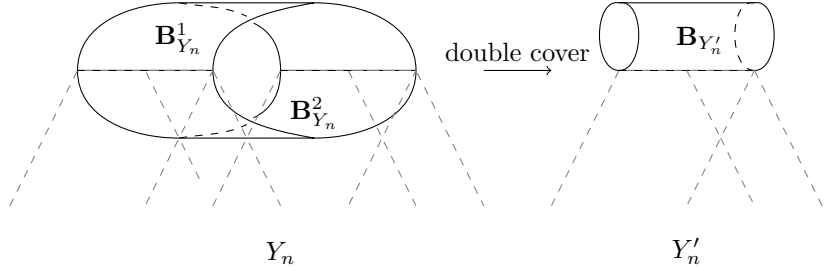
\begin{figure}[h]
   \centering
      \begin{tikzpicture}[scale=0.9]
      \footnotesize
\draw (0,3)to [out=90, in=180] (1.5,4)--(3.5,4)to [out=190,in=90](2,3)--(0,3); 
\draw (3.5,4) to [out=0, in=90] (5,3)--(3,3); 
\draw [dashed](1.5,4)to [out=-10, in=155] (2.5,3.8);
\draw  (3,3) to [out=90, in=-25] (2.6,3.7);
\draw (0,3)to [out=-90, in=-180] (1.5,2)--(3.5,2)to [out=-190,in=-90](2,3)--(0,3); 
\draw (3.5,2) to [out=0, in=-90] (5,3)--(3,3); 
\draw [dashed](1.5,2)to [out=10, in=-155] (2.5,2.2);
\draw  (3,3) to [out=-90, in=25] (2.6,2.3);
\draw[dashed] [gray] (-1,1)--(0,3)--(2,3)--(1,1);
\draw[dashed] [gray] (2,1)--(3,3)--(5,3)--(4,1);
\draw [dashed][gray] (1.8,1.4)--(1,3)--(2,3)--(3,1);
\draw [dashed][gray] (5,1)--(4,3)--(5,3)--(6,1);
\node [above] at (3,0) {$Y_n$};
\node [above] at (1.5,3.1) {$\BB_{Y_n}^1$};
\node [above] at (3.5,2) {$\BB_{Y_n}^2$};
\draw [->] (6,3)--(7,3);
\node [above]  at (6.5,3) {double cover};

\draw (8,3)--(10,3);
\draw (8,4)--(10,4);
\draw (8,3) to [out=0, in=10] (8,4);
\draw (8,3) to [out=180, in=170] (8,4);
\draw (10,3) to [out=0, in=10] (10,4);
\draw[dashed] (10,3) to [out=180, in=170] (10,4);
\draw [dashed][gray](7,1)--(8,3)--(10,3)--(9,1);
\draw [dashed][gray](11,1)--(10,3);
\draw [dashed] [gray](9,3)--(9.5,2);
\draw [dashed] [gray](9.5,2)--(10,1);
\node [above] at (9.2,3.1) {$\BB_{Y'_n}$};
\node [above] at (9,0) {$Y'_n$};
\end{tikzpicture}
\caption{\footnotesize $Y_n$ is a double cover of $Y'_n$. The preimages of a weight $0$ band $\BB_{Y'_n}$ are two pre-weight $0$ bands $\BB_{Y_n}^1$ and $\BB_{Y_n}^1$.} \label{F9}
\end{figure}

\begin{definition} \label{collapsebandnotwide}
   Let $Y$ be a union of bands and $\BB_Y\subset Y$ be a generalized band with bases $b_Y$ and $dual(b_Y)$.  We say a base $b\subset Y$ is \textbf{wide} if $b$ coincides with the block containing it, i.e. any other base intersecting $b$ is contained in $b$. $\BB_Y$ is \textbf{wide} if either $b_Y$ or $dual(b_Y)$ is wide. Suppose that $b_Y$ is wide, we may collapse $\BB_Y$ by first sliding all bases contained in $b_Y$ across $\BB_Y$, and then collapsing $\BB_Y$ from $b_Y$. This sequence of move is called \textbf{collapse a wide band}. 

      More generally, let $\BB_Y, \CC_Y \subset Y$ be two distinct generalized bands with the property that $b_Y \cap c_Y$ is non-degenerated (not a point). Using move $(M4)$ defined in Section~\ref{background} (which preserves $(A1)$ for $Y$), we can only slide $\BB_Y$ across $\CC_Y$ if $b_Y\subset c_Y$. However, if we further allow that a union of bands to have a simplicial forest as its underlying real graph, $\BB_Y$ can be slid across $\CC_Y$ as in Figure ~\ref{F8} even without the assumption that $b_Y\subset c_Y$. Let the block containing $b_Y$ and $c_Y$ be $K_{Y}$. 
In fact, we may slide every band with a base contained in $K_{Y}$ across $\CC_Y$. After all sliding,  $c_Y$ becomes a free subarc and so we collapse it. The new union of bands produced is the same as viewing $\CC_Y$ as an ultra thick segment of the underlying real graph of $Y$. This sequence of move is called \textbf{collapse a general band}.  
\end{definition}

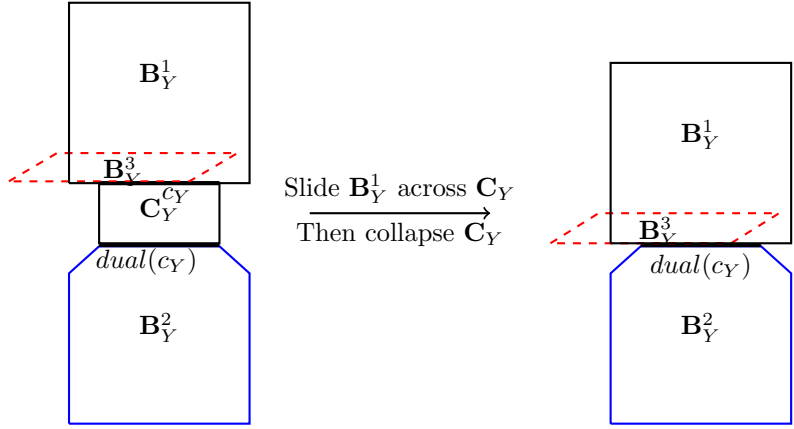
\begin{figure}[h]
   \centering
\begin{tikzpicture}[thick, scale=0.8]
\footnotesize
\draw [blue](9,-3.5)--(9,-1)--(9.5,-0.55)--(11.5,-0.55)--(12,-1)--(12,-3.5)--(9,-3.5);
\node [above] at  (10.5,-2.3) {$\BB_Y^2$};

\node [above] at  (10.5,0.9) {$\BB_Y^1$};
\draw[ultra thick] (9.5,-0.53)--(11.5,-0.53);
\node [below] at  (10.5,-0.5) {$dual(c_Y)$};

\draw[red] [dashed] (8,-0.5)--(8.8,0)--(11.8,0)--(11,-0.5)--(8,-0.5);
\node [below] at (9.8, 0.1) {$\BB_Y^3$};
\draw (9,-0.5)--(9,2.5)--(12, 2.5)--(12,-0.5)--(9,-0.5);

\draw [->] (4,0)--(7,0);
\node [above]  at (5.5,0) {Slide $\BB_Y^1$ across $\CC_Y$ };
\node [below]  at (5.5,0) {Then collapse $\CC_Y$};

\draw [blue](0,-3.5)--(0,-1)--(0.5,-0.55)--(2.5,-0.55)--(3,-1)--(3,-3.5)--(0,-3.5);
\node [above] at  (1.5,-2.3) {$\BB_Y^2$};

\draw(0.5,0.5)--(2.5,0.5)--(2.5,-0.5)--(0.5,-0.5)--(0.5,0.5);
\draw[ultra thick] (0.5,-0.53)--(2.5,-0.53);

\node [above] at  (1.5,-0.3) {$\CC_Y$};

\node [above] at  (1.5,1.9) {$\BB_Y^1$};
\draw[ultra thick] (0.5,0.5)--(2.5,0.5);
\node [left] at  (2.2,0.3) {$c_Y$};
\node [left] at  (2.3,-0.8) {$dual(c_Y)$};

\draw[red]  [dashed](-1,0.53)--(-0.2,1)--(2.8,1)--(2,0.53)--(-1,0.53);
\node [below] at (0.9, 1.1) {$\BB_Y^3$};

\draw (0,0.5)--(0,3.5)--(3, 3.5)--(3,0.5)--(0,0.5);
\end{tikzpicture}
\caption{\footnotesize We may slide $\BB_Y^1$ across $\CC_Y$ even if $c_Y$ does not contain $b_Y$. If the red dashed $\BB_Y^3$ exists, $\BB_Y^3$ can also be slided as shown}\label{F8}
\end{figure}

We are now ready to describe Process III which again will produce from $(X\overset{\i}{\to}X')$ another pair of band complexes $(X^* \overset{\i^*}{\to}X')$ ($X'$ remains the same). Process III will only be applied when $Y'_0 $ is in a standard form and $Y_0$ contains some pre-weight $0$ or pre-weight $\frac{1}{2}$ bands. 
\vspace{0.5pc}

Suppose that $Y_0$ contains in total $N_w$ ($N_w>0$) pre-weight $0$ and pre-weight $\frac{1}{2}$ bands. Let $\BB_{Y}\subset Y_0$ be one of them and $\BB_{Y'}\subset Y'_0$ be the weight $0$ (or $\frac{1}{2}$) band containing $\i_0(\BB_{Y})$. Further let the block of $\Gamma_{Y_0}$ containing $b_{Y}$ be $K_{b_{Y}}$ and the block of $\Gamma_{Y'_0}$ containing $b_{Y'}(=dual(b_{Y'}))$ be $K_{b_{Y'}}$. Apply either a \textit{collapse a wide band move} or a \textit{collapse a general band move} to collapse $\BB_{Y}$
(i.e. slide all bands with the property that one of whose bases is contained in $K_{b_{Y}}$, across $\BB_{Y}$ and then collapse $\BB_{Y}$). Denote the resulting component by $Y^*_0$. Let the block (may not be an edge) in $Y^*_0$ formed in this move be $K^*_{b_{Y}}$. Since $\i_0(b_{Y})=\i_0(dual(b_{Y}))$, $\i_0: Y_0\to Y'_0$ induces a well-defined morphism $\i^*_0: Y^*_0\to Y'_0$. In particular, $\i^*_0:K^*_{b_{Y}} \to K_{b_{Y'}}$ is a morphism between graphs. 
Let $X^*$ be the resulting band complex corresponding to $Y^*_0$ and $\i^*: X^*\to X'$ be the morphism that equals to $\i^*_0$ when restricting to $Y^*_0$ and remains the same as $\i$ otherwise. Further, according to Proposition ~\ref{maimmersion}, up to a finite folding sequence, we may assume that $\i^*$ is an immersion. $(X^*\overset{\i^*}{\to}X')$ is then the new pair of band complexes produced from $(X\overset{\i}{\to}X')$ by Process III.

Let the total number of pre-weight $0$ and pre-weight $\frac{1}{2}$ bands in $X^*$ be $N_w^*$. It is clear that $N_w^*\leq N_w-1$ since at least $\BB_Y$ is collapsed. The inequality holds if there are further folding of bands. If other than $b_Y$, there are more pre-weight $0$ or $\frac{1}{2}$ bases in $K_{b_{Y}}$, they may become weight $0$ or $\frac{1}{2}$ in $Y^*_0$. If $N_w^*=0$, we say Process III sequence ends, and then go back to Process I with $(X^*\overset{\i^*}{\to}X')$. Otherwise, we are in the position to apply Process III again.

\begin{prop}\label{finalrelation}
   Let $(Y_0\overset{\i_0}{\to}Y'_0)$ be a pair of components and  $(Y_0\overset{\i_0}{\to}Y'_0), (Y_1\overset{\i_1}{\to}Y'_1),\dots$ be a sequence formed by the successively application of the relative Rips machine. Process III appears only finitely many times in this sequence. In particular, there exists $N>0$ such that for any $ n>N$,  $Y_n'$ is in a standard form,  $Y_n$ contains no pre-weight $0$ nor pre-weight $\frac{1}{2}$ band and $\i_n: Y_n\to Y'_n$ is an almost partial covering map when $Y'_0$ is a surface or thin component. 
\end{prop}

\begin{proof}
Suppose that $(Y_{i+1}\overset{\i_{i+1}}{\to}Y'_{i+1})$ is obtained from $(Y_i\overset{\i_i}{\to}Y'_i)$ by Process III. Then the total number of pre-weight $0$ and pre-weight $\frac{1}{2}$ bands in $Y_{i+1}$ is less than the number of $Y_i$ due to the collapse of a pre-weight $0$ or pre-weight $\frac{1}{2}$ band. Further possibly some pre-weight $0$ or $\frac{1}{2}$ bands become weight $0$ or $\frac{1}{2}$ bands and folds between bands may also be needed. Let $(Y_k\overset{\i_k}{\to}Y'_k)\to (Y_{k+1}\overset{\i_{k+1}}{\to}Y'_{k+1})$ (note $Y'_{k+1}=Y'_k$) be the first time that Process III is applied. Then after finitely many steps, say $m$ steps, the machine will return to Process~I. Furthermore, $Y'_{k}$ is in a standard form (otherwise Process III would not be applied) implies that any weight $0$ (resp. weight $\frac{1}{2}$) band in $Y'_{k+i}$ is a sub-generalized band of a weight $0$ (resp. weight $\frac{1}{2}$) generalized band in $Y'_{k}$ for any $i>0$.  As consequence, that $Y_{k+m}$ contains no pre-weight $0$ nor pre-weight $\frac{1}{2}$ band implies $Y_{k+m+i}$ contains no pre-weight $0$ nor pre-weight $\frac{1}{2}$ band for any $i>0$. Thus, Process III won't be applied again. Let $N=k+m$, and we are done.
\end{proof}

\section{Machine Output}\label{machineoutput}
\subsection{Special case}\label{special}
For a pair of components $(Y_0\overset{\i_0}{\to}Y'_0)$, in the relative Rips machine, eventually either only Process I is applied ($Y'_0$ is of thin type) or only Process II is applied ($Y'_0$ is of surface or toral type). In this section, for a fixed minimal component $Y'_0$, we will examine possible outputs for the structure of $Y_0$. 
\vspace{1pc}

When one of $\{Y_0, Y'_0\}$ is a toral component, the type of the other one can be quickly determined according to Lemma~\ref{subtoral}. Therefore our following discussion will focus more on the case of surface and thin components.

\begin{lemma}\label{subtoral}
   Let $(Y_0\overset{\i_0}{\to}Y'_0)$ be a pair of components with $\i_0$ an immersion, then 
   \begin{enumerate}
      
      \item If $Y_0$ is a toral component with rank $n>2$, then so is $Y'_0$.
      \item If $Y'_0$ is a toral component and $Y_0$ is not a simplicial component, then $Y_0$ is also a toral component .
   \end{enumerate}

\end{lemma}

\begin{proof}
(1) $Y_0$ is a toral component implies that there are infinitely many points in the limit graph of $Y_0$ (Definition ~\ref{index}) having positive indexes. This must also hold for $Y'_0$. So $Y'_0$ is a toral  component.

       (2) $Y'_0$ is a toral component implies its dual tree is a line. Therefore the minimal subtree corresponding to $Y_0$ is also a line. Thus $Y_0$ is a toral component. 
     \end{proof}

\begin{lemma}\label{simplcialimage}
Let $(Y_0\overset{\i_0}{\to}Y'_0)$ be a pair of components. Suppose that $Y'_0$ is a simplicial component, then $Y_0$ is also a simplicial component. In particular, for any pair of components $(Y_0\overset{\i_0}{\to}Y'_0)$, if $\widehat{Y}_0=\i_0(Y_0)$ is simplicial, then so is $Y_0$.
\end{lemma}

\begin{proof}
$Y'_0$ is simplicial implies $\widehat{Y}_0$ is also simplicial. Suppose that $Y_0$ is minimal, then every leaf in $Y_0$ is dense, which implies its image is also dense in $\widehat{Y}_0$. This contradicts to that $\widehat{Y}_0$ is simplicial. 
\end{proof}

\vspace{1pc}
For a given pair of components $(Y_0\overset{\i_0}{\to}Y'_0)$, successive application of the relative Rips machine will eventually convert $\i_0$ into an almost partial covering map. The best scenario is that $\i_0$ is in fact a finite covering map. The type of $Y_0$ is then determined by Lemma~\ref{cover}. For the general case, we will first complete $Y_0$ into a finite cover as in Lemma~\ref{completecover} and Lemma~\ref{completepartialcover}, then use this constructed finite cover to study the original pair.

\begin{lemma}\label{cover}
   Let $(Y_0\overset{\i_0}{\to}Y'_0)$ be a pair of components with $\i_0$ an immersion. If further $\i_0: Y_0\to Y'_0$ is also a submersion, i.e. $\i_0$ is a local isomorphism, then $\i_0:Y_0\to Y'_0 $ is a covering map of finite degree. In particular, $Y_0$ and $Y'_0$ are of the same type. 
\end{lemma}

\begin{proof} 
It is easy to check that $\i_0$ is a finite covering map, see for example in \cite{hatcher}. By Proposition~\ref{character}, Lemma~\ref{subtoral} and Lemma~\ref{simplcialimage} we have the following. If one of $\{Y_0,Y'_0\}$ is simplicial, then both of them are simplicial. If one of $\{Y_0,Y'_0\}$ is toral, then both of them are toral. Moreover if one of $\{Y_0,Y'_0\}$ is a surface component, then all but finitely many points in $\Gamma_{Y'_0}$ and $\Gamma_{Y_0}$ are of zero indexes, thus both of them are surface components. If one of $\{Y_0,Y'_0\}$ is a thin component, then the limit graphs of $Y_0$ and $Y'_0$ are dense $G_{\delta}$-sets, and so both of them are thin components.
\end{proof}



\begin{cor}\label{standardcover}
   Let $(Y_0\overset{\i_0}{\to}Y'_0)$ be a pair of components. Suppose that $Y'_0$ is a minimal component in a standard form and $Y_0$ is a finite cover of $Y'_0$. Further assume that $Y_0$ contains no pre-weight $0$ bands nor pre-weight $\frac{1}{2}$ bands. Then $Y_0$ is also in a standard form. 
\end{cor}

\begin{proof}
By Lemma ~\ref{cover}, $Y_0$ is a minimal component of the same type as $Y'_0$. 
The assumption that $Y_0$ contains no pre-weight $0$ band nor pre-weight $\frac{1}{2}$ band implies that every generalized band $\BB_{Y_0}$ in $Y_0$ and its image $\BB_{Y'_0}$ in $Y'_0$ have the same weight. In particular, 
for any $q\in \Gamma_{Y_0}$ and its image $\i_0(q)=p \in \Gamma_{Y'_0}$, $i_{Y_0} (q)=i_{Y'_0} (p)$, i.e. $p$ and $q$ have the same index. Since $Y'_0$ is in a standard form, by Proposition~\ref{character}, $Y_0$ is also in a standard form. 
\end{proof}





Omitting weight $0$ bands of a component does not change the type of that component.
For a pair of components $(Y_0\overset{\i_0}{\to}Y'_0)$, suppose that $\i_0$ is a local isomorphism up to weight $0$ bands. The following Lemma \ref{coveruptozero} shows that a similarly conclusion as in Lemma \ref{cover} also holds in this case.

\begin{lemma}\label{coveruptozero}
   Let $(Y_0\overset{\i_0}{\to}Y'_0)$ be a pair of components. Suppose that $\i_0$ is a local isomorphism up to weight $0$ bands. Then $\i_0:Y_0\to Y'_0$ can be extended into a finite covering map $\i'_0: Y^*_0 \to Y'_0$ where $Y_0\hookrightarrow Y^*_0$ is a union of bands. Moreover, we may construct $Y^*_0$ by attaching finitely many bands to $Y_0$ such that all of these attached bands are either weight $0$ bands or may be converted into weight $0$ bands after an application of the Rips machine. In particular, $Y_0$ is a minimal component of the same type as $Y'_0$.  
\end{lemma}

\begin{proof}
We may assume that $\i_0$ is a partial covering map (i.e. In $Y'_0$, the image of each weight $0$ or pre-weight $0$ band of $Y_0$ is a generalized band, not a proper sub-generalized band). Otherwise, we may archive this by subdividing weight $0$ bands in $Y'_0$ and their preimages in $Y_0$ correspondingly. For a given weight $0$ band $\BB_{Y'_0}$ with base $b_{Y'_0}$ in $Y'_0$, let $b_{Y_0}^1, \dots, b_{Y_0}^n\subset \Gamma_{Y_0}$ be preimages of $b_{Y'_0}$ in the real graph of $Y_0$. Every connected component of the union of preimages of $\BB_{Y'_0}$ in $Y_0$ is either a weight $0$ band or a consecutive sequence of pre-weight $0$ bands. 

If $\i_0$ is in fact a local isomorphism, let $Y^*_0=Y_0$ and we are done by Lemma~\ref{cover}. Otherwise there exists some weight $0$ band $\BB_{Y'_0}$ in $Y'_0$ such that $\i_0$ is not surjective near some $b_{Y_0}^k$. Since $\i_0$ is a partial covering map, there are only two possibilities. One is that no band in $Y_0$ with base $b_{Y_0}^k$ maps to $\BB_{Y'_0}$. We may fix this by adding a weight $0$ band $\BB^k_{Y_0}$ to $Y_0$ along $b^k_{Y_0}$ and defining $\i^*_0$ to be the map that maps $\BB^k_{Y_0}$ onto $\BB_{Y'_0}$. 

The other case is that there exits exactly one pre-weight $0$ band $\BB$ with a base $b^k_{Y_0}$ mapping onto $\BB_{Y'_0}$. Then $\BB$ must be one end of a consecutive sequence of pre-weight $0$ bands mapping onto $\BB_{Y'_0}$. One end base of this sequence is $b_{Y_0}^k$ and the other end base must be one of $\{b_{Y_0}^1, \dots, b_{Y_0}^n\}$, say $b^{k'}_{Y_0}$. We may fix this situation by adding a band $\BB'$ to $Y_0$ with one of its bases attached along $b^k_{Y_0}$ and the other base attached along $b^{k'}_{Y_0}$ (orient in the same direction) and defining $\i^*_0$ to be the map that maps $\BB'$ onto $\BB_{Y'_0}$. 

Apply above attaching rule to every such $b^k_{Y_0}$ (for a fixed $\BB_{Y'_0}$) and for every weight $0$ band in $Y'_0$. Let the resulting union of bands be $Y^*_0$. $\i^*_0: Y^*_0\to Y'_0$ is a local isomorphism by construction and so $Y^*_0$ is a finite covering space of $Y'_0$. Moreover in $Y^*_0$, we may alter each $\BB'$ to a weight $0$ band by sliding one of its bases across the consecutive sequence of pre-weight $0$ bands containing $\BB$. Denote the resulting union of bands by $Y^{**}_0$. Since $Y_0$ only differs from $Y^{**}_0$ by weight $0$ bands,  $Y_0$ has the same type as $Y^{**}_0$, and further the same as $Y^*_0$ and $Y'_0$. 
\end{proof}

\vspace{1pc}
In the proof of Lemma~\ref{coveruptozero}, for a union of bands $Y'_0$, we started with a special partial cover $Y_0$ and completed it into a finite cover $Y^*_0$. More generally, mimicking the argument for completing a partial cover of graphs in \cite{stal83}, the following lemma shows that a similar completing process can be done for any partial cover of band complexes.

\begin{lemma}\label{completecover}
   Let $(Y_0\overset{\i_0}{\to}Y'_0)$ be a pair of components with $\i_0:Y_0\to Y'_0$ a partial covering map. Then by adding finitely many arcs to the real graph $\Gamma_{Y_0}$ (using move $(M7)$) and attaching finitely many new bands to $Y_0$ (using move $(M6)$), we may extend $\i_0$ into a finite covering map. Moreover, assume that $Y_0$ contains no pre-weight $\frac{1}{2}$ bands (resp. pre-weight $0$ bands), then the newly constructed $Y_0$ also contains no pre-weight $\frac{1}{2}$ bands (resp. pre-weight $0$ bands). In particular, on the level of fundamental groups, $\pi_1(Y_0)$ is either a finite index subgroup of $\pi_1(Y'_0)$ or a free factor of a finite index subgroup of $\pi_1(Y'_0)$. 
\end{lemma}

\begin{proof}

$\Gamma_{Y_0}$ and $\Gamma_{Y'_0}$ contain only finitely many blocks. Note that $\i_0$ is a finite covering map is equivalent to that $\i_0$ is a local isomorphism near every block. $Y'_0$ and $Y_0$ can be viewed as graphs by considering each block as a vertex and each generalized band as an edge. $\i_0$ then can be considered as an immersion between finite graphs with the property that each edge in $Y_0$ is mapped exactly onto one edge in $ Y'_0$.

The preimages of a block $K \subset \Gamma_{Y'_0}$ are finitely many blocks $K_1,\dots,$ $ K_{n(K)}$ in $\Gamma_{Y_0}$. Different blocks in $\Gamma_{Y'_0}$ may have different numbers of preimages ($n(K)$ depends on $K$). Let $n=\max\{n(K)\}_K$. For blocks in $\Gamma_{Y'_0}$,  whose number of preimages is less than $n$, add extra arcs (move $M7$) to $\Gamma_{Y'_0}$ to complete $\i_0$ as a covering map on the level of real graphs. Then we need to complete $\i_0$ for the level of bands. Let $\BB_{Y'_0}\subset Y'_0$ be a generalized band. If $b_{Y'_0}$ and $dual(b_{Y'_0})$ are contained in the same block, processed exactly as in Lemma~\ref{coveruptozero}. Otherwise, let the block containing $b_{Y'_0}$ be $K$, the block containing $dual(b_{Y'_0})$ be $K'$ and $K\neq K'$. At each preimage $K_i$ of $K$, there is either no band or exactly one band $\BB^i_{Y_0}$ maps to $\BB_{Y'_0}$ since $\i_0$ is an immersion. The band $\BB^i_{Y_0}$ uniquely determines one preimage of $K'$. Thus each preimage of $\BB_{Y'_0}$ in $Y_0$ groups one block in $\{K_1, \dots, K_n\}$ with one 
block in $\{K'_1,\dots, K'_n\}$. After paring up like this, there are the same numbers of $\{K_i\}$ and $\{K'_i\}$ left unpaired. We may then pair up these left blocks randomly. For each of these new pairs, $(K_i, K'_j)$, attach a new band $\BB^*_{Y_0}=b^*_{Y_0}\times I_N$ with the property that $m(b^*_{Y_0})=m(b_{Y'_0})$ and $N=l(\BB_{Y'_0})$ to $Y_0$ in the following way:
glue $b^*_{Y_0}$ to $b_{Y'_0}$'s preimage in $K_i$ and glue $dual (b^*_{Y_0})$ to $dual (b_{Y'_0})$'s preimage in   $K'_i$. Further extend $\i_0$ into $\BB^*_{Y_0}$ by mapping it to $\BB_{Y'_0}$. Do the above process for all generalized bands in $Y'_0$. It is easy to check that the resulting union of band is then a finite cover of $Y'_0$. 

In the case that $Y_0$ contains no pre-weight $\frac{1}{2}$ (resp. pre-weight $0$) band,  the set of preimages of weight $\frac{1}{2}$ (resp. weight $0$) bands in $Y'_0$ contains only weight $\frac{1}{2}$ (resp. weight $0$) bands in $Y_0$. Restricting $\i_0$ to the union of weight $\frac{1}{2}$ (resp. weight $0$) bands in $Y'_0$, $Y_0$ can be convert to a finite covering by adding only weight $\frac{1}{2}$ (resp. weight $0$) bands. Thus the new constructed $Y_0$ contains no pre-weight $\frac{1}{2}$ (resp. pre-weight $0$) band.
\end{proof}






In the proof of Proposition~\ref{SOF}, we will need to complete an almost partial cover into a cover. The following lemma shows that we may first complete an almost partial cover into a partial cover and then, using Lemma~\ref{completecover}, complete the partial cover into a finite cover. 

\begin{lemma} \label{completepartialcover}
Let $(Y_0\overset{\i_0}{\to}Y'_0)$ be a pair of components and $\i_0:Y_0\to Y'_0$ be an almost partial covering map. Then there exists a pair of components $(Y^*_0\overset{\i^*_0}{\to}Y'_0)$ with the property that $\i^*_0$ is a partial covering map, $Y_0\subset Y^*_0$ and $\i^*_0|_{Y_0}=\i_0$. In particular, $Y^*_0$ can be constructed by adding finitely many arcs to the real graphs $\Gamma_{Y_0}$ and attaching finitely many bands to $Y_0$. 
\end{lemma}

\begin{proof}
Let $\BB_{Y'_0}$ be a generalized band in $Y'_0$ and $\BB_{Y_0}$ be a generalized band in $Y_0$ that maps properly into (not onto) $\BB_{Y'_0}$. Further, let $b_{Y_0}$ and $dual(b_{Y_0})$ be bases of $\BB_{Y_0}$, $b_{Y'_0}$ and $dual(b_{Y'_0})$ be bases of $\BB_{Y'_0}$. Then at least one of the endpoints of $b_{Y_0}$, denote it by $q$, maps to an interior point $p\in b_{Y'_0}$. $\i_0$ is an almost partial covering map implies that there is no other band with its base(s) contained in the same block as $b_{Y_0}$ or $dual(b_{Y_0})$ mapping into $\BB_{Y'_0}$. Otherwise that band must have some overlap with $\BB_{Y_0}$ which contradicts to $\i_0$ is an immersion. Further we may assume that within the block containing $b_{Y_0}$ (resp. $dual(b_{Y_0})$), there is a segment $b^*_{Y_0}$ (resp. $dual(b^*_{Y_0})$) maps onto $b_{Y'_0}$ (resp. $dual(b_{Y'_0})$). Otherwise, we may obtain $b^*_{Y_0}$ or $dual(b^*_{Y_0})$ by adding new arcs to $\Gamma_{Y_0}$ ($\gamma'$ in Figure~\ref{F6}). Then by attaching a new band to $Y_0$, $\BB_{Y_0}$ can be extended into a new band whose bases are $b^*_{Y_0}$ and $dual({b}^*_{Y_0})$.
Let the resulting band be $\BB^*_{Y_0}$. We then define $\i^*_0$ be the same as $\i_0$ on $Y_0$ and maps $\BB^*_{Y_0}$ onto $\BB_{Y'_0}$. Finally since there are only finitely many generalized bands, $Y_0$ can be completed into a finite partial cover $Y^*_0$ within finite steps. 
\end{proof}

\begin{figure}[h]
\centering
\begin{tikzpicture}[thick, scale=0.9]
\footnotesize
\draw [fill=lightgray](0,-2)--(0,0)--(2,0)--(2,-2)--(0,-2);
\draw (0,0)--(2,0)--(2, 2)--(0,2)--(0,0);

\draw [fill=lightgray] (0.5,0)--(0.5,2)--(2, 2)--(2,0)--(0.5,0);

\coordinate (q) at (0.5,0);
\filldraw(q) circle (2pt);
\node [above] at (0.4,0) {$p$};
\node [right] at (2, 0) {$b_{Y'_0}$};
\node [above] at (1, 0.8) {$\BB_{Y'_0}$};
\node [right] at (2, 2) {$dual(b_{Y'_0})$};
\node [left] at (-1,0) {$Y'_0$:};

\draw [fill=lightgray] (7,-2)--(7,0)--(9,0)--(9,-2)--(7,-2);

\draw [fill=lightgray] (7,0)--(9,0)--(9, 2)--(7,2)--(7,0);
\node [above] at (8, 0.8) {$\BB_{Y'_0}$};
\end{tikzpicture}

\begin{tikzpicture}[thick, scale=0.9]
\footnotesize
\draw (0,-2)--(0,0)--(2,0)--(2,-2)--(0,-2);
\draw  (0.5,0)--(0.5,2)--(2, 2)--(2,0)--(0.5,0);

\coordinate (p) at (0.5,0);
\filldraw(p) circle (2pt);
\node [above] at (0.4,0) {$q$};
\draw [red] (0,-0.05)--(2,-0.05);
\node [right][red] at (2, 0) {$b^*_{Y_0}$};
\node [above] at (1.25,2) {$dual(b_{Y_0})$};
\node [above] at (1.25,0) {$b_{Y_0}$};
\node [above] at (1.25, 0.8) {$\BB_{Y_0}$};
\node [left] at (-1,0) {$Y_0$:};

\draw [->] (1,2.8)--(1,3.5);
\draw [->]  (4,1)--(6,1);
\draw [->](8, 2.8)--(8, 3.5);

\draw  (7,-2)--(7,0)--(9,0)--(9,-2)--(7,-2);
\draw  (7,0)--(9,0)--(9, 2)--(7,2)--(7,0);
\node [above] at (8,0.8) {$\BB_{Y_0}$};

\draw [ultra thick] (7,2)--(7.5,2);
\node [above] at (7.25,2) {$\gamma'$};
\end{tikzpicture}
\caption{\footnotesize The gray part within $Y'_0$ is $\widehat{Y}_0$.} \label{F6}      
\end{figure}
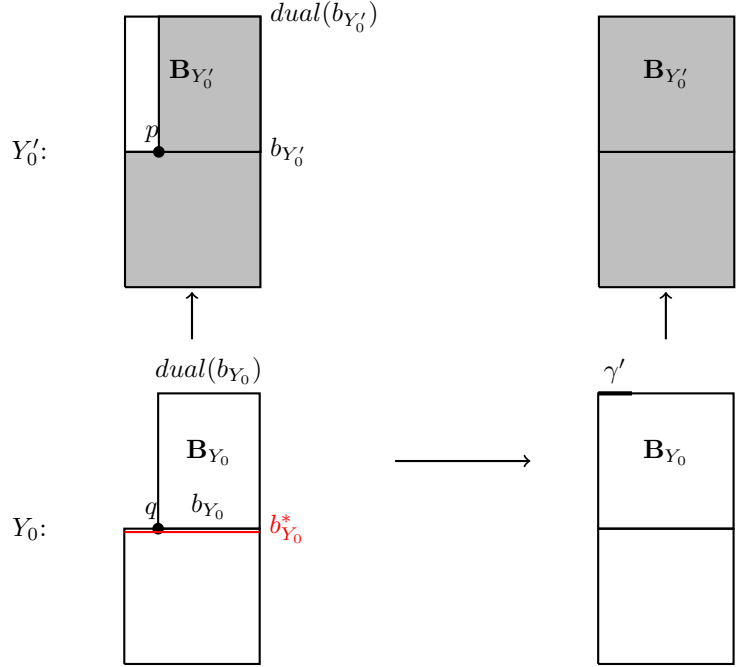

For a pair of components $(Y_0\overset{\i_0}{\to}Y'_0)$, we will now start to discuss the type of $Y_0$ when $Y_0$ is not a finite cover of $Y'_0$. In Lemma~\ref{notonto}, we will show that if $Y'_0$ is a surface or thin component and the relative Rips machine does not convert $\i_0$ into a map which is surjective (i.e. $ \widehat{Y}_0\subsetneq Y'_0$), then $Y_0$ must be simplicial. In Proposition~\ref{SOF}, we will discuss the case where $\i_0$ is surjective but not locally surjective.

\begin{lemma}\label{notonto}
   Let $(Y_0\overset{\i_0}{\to}Y'_0)$ be a pair of components with $\i_0$ an immersion. Further assume that $Y'_0$ is either of thin type or surface type and it is in a standard form. If there exists a point $p$ in $Y'_0$ with $i_{Y'_0}(p)\geq 0$ such that it has no pre-image in $Y_0$, then $Y_0$ is simplicial. In particular, the fundamental group of $Y_0$ is an infinite index subgroup of the fundamental group of $Y'_0$. 
\end{lemma}

\begin{proof}
 Let $\widehat{Y}_0=\i_0(Y_0)$. By Lemma~\ref{simplcialimage}, we only need to show that $\widehat{Y}_0$ is simplicial. 

$\widehat{Y}_0$ is a band sub-complex (a closed set) in $Y'_0$. By the assumption, $p$ is contained in the complement of $ \widehat{Y}_0$ in $Y'_0$. So there exists a neighborhood of $p$ that is contained in this complement. Without loss, we may assume that $p$ is an interior point of some band in $Y'_0$. Hence, there exists a sub-generalized band $\BB_{Y'_0}$ in $Y'_0$ containing $p$ which is disjoint from $\widehat{Y}_0$.

Assume that $\widehat{Y}_0$ is not simplicial, then by Lemma~\ref{subtoral}, $\widehat{Y}_0$ is either a surface component or a thin component. According to Proposition~\ref{character}, since $Y'_0$ is in a standard form, there is a collection of uncountably many leaves in $\widehat{Y}_0$ which are $2$-ended trees. In particular, they are quasi-isometric to lines \footnote{In the case of surface, all but finitely many of leaves are q.i. to lines, see \cite[section 8]{bfstableactions}.}. Further in $Y'_0$, leaves containing these leaves are also $2$-ended trees expect for finitely many. Therefore we may pick a leaf $l$ in $\widehat{Y}_0$ which is a $2$-ended tree such that the leaf $l'$ in $Y'_0$ containing it is also a $2$-ended tree. $l'$ is quasi isometric to a line contained in the limit set. Let the line be $\overline{l}'$($=l'\cap \Omega_{Y'_0}$). $i_{Y'_0}(p)\geq 0$ implies that $\overline{l}'$ intersects $\BB_{Y'_0}$ infinitely often. Hence, each component of $l'\cap \widehat{Y}_0$ is either a finite tree or an $1$-ended tree. This contradicts to our choice that $l$ is a $2$-ended tree. Thus $\widehat{Y}_0$ must be simplicial, and so is $Y_0$.

The immersion $\i_0: Y_0\to Y'_0$ induces a monomorphism between their fundamental groups. Since $Y_0$ is simplicial, the minimal translation length of elements in the fundamental group of $Y_0$ is bounded below by some positive number $\epsilon$. On the other hand, $Y'_0$ is a minimal component, so there exists sequence of elements in the fundamental group of $Y'_0$ whose translation lengths converge to $0$. Therefore, $\pi_1(Y_0)$ is an infinite index subgroup of $\pi_1(Y'_0)$. 
\end{proof}

\begin{prop}\label{SOF}
   Let $(Y_0\overset{\i_0}{\to}Y'_0)$ be a pair of components. 
   Suppose that $Y'_0$ 
   is a surface or thin component. 
   Then either $Y_0$ is simplicial or omitting weight $0$ bands, the relative Rips machine eventually converts $Y_0$ into a finite cover of $Y'_0$.
\end{prop}

\begin{proof}
By Proposition~\ref{finalrelation}, we may assume that $Y'_0$ is in a standard form, $\i_0$ is an almost partial covering map and $Y_0$ contains no pre-weight $0$ bands nor pre-weight $\frac{1}{2}$ bands. In particular, $\i_0(\overline{Y}_0)\subset \overline{Y}'_0$. Therefore we may ignore weight $0$ bands and only work with $(\overline{Y}_0 \overset{\i_0}{\to} \overline{Y}'_0)$ instead. To simplify the notation, without loss, we will assume that both $Y_0$ and $Y'_0$ do not contain any weight $0$ bands. 

First $(Y_0\overset{\i_0}{\to}Y'_0)$ can be extended into $(Y^*_0\overset{\i^*_0}{\to}Y'_0)$ such that ${\i}^*_0$ is a partial covering map as in Lemma~\ref{completepartialcover}. Then by Lemma~\ref{completecover}, $\i^*_0$ can further be extended into a finite covering map $\widetilde{\i}^*_0: \widetilde{Y}^*_0\to Y'_0$. Lemma~\ref{cover} implies that $ \widetilde{Y}^*_0$ has the same type as $Y'_0$ and Corollary~\ref{standardcover} shows that $ \widetilde{Y}^*_0$ is also in a standard form.

For $Y_0\hookrightarrow Y_0^*\hookrightarrow  \widetilde{Y}^*_0 $, we may assume that $Y_0\subsetneq \widetilde{Y}^*_0$, or we are done. If the image of $Y_0$ in $ \widetilde{Y}^*_0$ misses any point in the limit set of $ \widetilde{Y}^*_0$, $Y_0$ is simplicial by Lemma~\ref{notonto}. Otherwise, each 
sub-generalized band $\CC \in \widetilde{Y}^*_0-Y_0$ does not contain any point in the limit set of $\widetilde{Y}^*$ and so image of $\CC$ in $Y'_0$ must also do not contain any point in the limit set of $Y'_0$. In this case, $ \widetilde{Y}^*_0$ must be a thin type component since all but finitely many leaves of a surface component are in the limit set.

By construction, the following diagram commutes.

$$\begindc{\commdiag}[20]
\obj(1,1)[Y0]{$Y_0$}
\obj(3,1)[Y0*]{$Y_0^*$}
\obj(5,1)[Y0']{$\widetilde{Y}_0^*$}
\obj(1,3)[X0]{$Y'_0$}
\mor{Y0}{Y0*}{}[\atright,\injectionarrow]
\mor{Y0*}{Y0'}{}[\atright,\injectionarrow]
\mor{Y0}{X0}{$\i_0$}
\mor{Y0'}{X0}{$\widetilde{\i}^*_0$}[\atright,\solidarrow]
\enddc  $$

Now consider three pairs of components $(Y_0\overset{\i_0}{\to}Y'_0)$, $(Y_0\hookrightarrow \widetilde{Y}_0^*)$ and $(\widetilde{Y}_0^*\overset{\widetilde {\i}^*_0}{\to} Y'_0)$. 
For any given sequence $(Y_0\overset{\i_0}{\to}Y'_0), (Y_1\overset{\i_1}{\to}Y'_1),\dots$ formed by the relative Rips machine, the procedure of obtaining $(Y_1\overset{\i_1}{\to}Y'_1)$ from $(Y_0\overset{\i_0}{\to}Y'_0)$ can be viewed as a combination of obtaining $(\widetilde{Y}_1^*\overset{\widetilde {\i}^*_1}{\to} Y'_1)$ from $(\widetilde{Y}_0^*\overset{\widetilde {\i}^*_0}{\to} Y'_0)$ and then obtaining $(Y_1\hookrightarrow \widetilde{Y}_1^*)$ from $(Y_0\hookrightarrow \widetilde{Y}_0^*)$. 

Let $\widetilde{Y}^*_i$ be the induced intermediate finite cover of $Y'_i$.
By above analysis, images of $\CC$'s in $Y'_0$ does not contain any point in the limit set and so fully collapse within finitely many steps (Proposition ~\ref{thinprop}). As a consequence, $\CC$'s in $\widetilde{Y}^*_0$ also fully collapse within finitely many steps. Therefore $Y_n=\widetilde{Y}^*_n $ for sufficiently large $n$. Thus eventually $Y_n$ is a finite cover of $Y'_n$. 
\end{proof}

Following immediately from Proposition ~\ref{finalrelation}, Lemma ~\ref{subtoral}, Lemma ~\ref{cover} and Proposition~ \ref{SOF}, we have: 

\begin{thm}\label{t:main}
Let $(X\overset{\i}{\to}X')$ be a pair of band complexes and $(Y_0\overset{\i_0}{\to}Y'_0)\subset (X\overset{\i}{\to}X')$ be a pair of components. Suppose that $Y'_0$ is a minimal component, then $Y_0$ is either a minimal component of the same type as $Y'_0$ or $Y_0$ is simplicial. Moreover, $Y_0$ and $Y'_0$ are both surface or thin components if and only if the relative Rips machine eventually converts $(Y_0 \overset{\i} {\to} Y'_0)$ into a pair $(Y_0^* \overset{\i^*} {\to} Y_0^{'*})$ with the property that $\overline{Y_0^*}$ is a finite cover of $\overline{Y_0^{'*}}$.


\end{thm}

\begin{remark}\label{e:infinitetoralcover}
It is possible that both $Y_0$ and $Y'_0$ are of toral type but the relative Rips machine does not convert $Y_0$ into a finite cover of $Y'_0$ omitting weight $0$ bands. For example, let $Y_0$ be a toral component that is dual to the action of $\Z^{3}=\langle a\rangle \times \langle b\rangle \times \langle c\rangle $ on a real line where $a\to 1, b\to e, c\to \pi$. Its infinite index subgroup $\langle a\rangle \times \langle b\rangle \simeq \Z^2$ also acts on the real line and is dual to a toral component. 
\end{remark}



\begin{cor}\label{output2}
Let $H<G$ be two finitely presented groups. Further let $T_G$ be a $G$-tree with trivial edge stabilizers and $T_H\subset T_G$ be a minimal $H$-subtree. Suppose that $(X\overset{\i}{\to}X')$ is a pair of band complexes, that $X$ and $X'$ resolve $T_H$ and $ T_G$ correspondingly, that $Y$ and $Y'$ are single minimal components of either surface or thin type and  that $\pi_1(Y)$ generates $H$, $\pi_1(Y')$ generates $G$. Then $[G:H]$ is finite. 
\end{cor}


\begin{proof}
Apply the relative Rips machine to $(X\overset{\i}{\to}X')$. Since $Y$ is a surface or thin component, according to Theorem \ref{t:main}, the machine will eventually convert $(X\overset{\i}{\to}X')$ into a new pair $(X^*\overset{\i}{\to}{X'}^*)$ such that $\overline{Y}^*$ is a finite cover of $\overline{Y}'^*$. Since $T_G$ has trivial edge stabilizers, we may assume that $Y^*$ and ${Y'}^*$ contain no weight $0$ bands (\cite[section 8]{bfstableactions}). Thus $Y^*$ is a finite cover of ${Y'}^*$.
       
 Claim that $X^*$ is also a finite cover of ${X'}^*$. We may assume that $X^*$ and ${X'}^*$ are obtained from $Y^*$ and ${Y'}^*$ by attaching $2$-cells due to the fact that $\pi_1(Y)$ generates $H$ and $\pi_1(Y')$ generates $G$. 
 Because $H$ is a subgroup of $G$. If a $2$-cell is attached along a loop $l$ in $Y^*$, there must be a corresponding $2$-cell is attached along the image of $l$ in ${Y'}^*$. 
 Moreover every loop in ${Y'}^*$ has finitely many preimages in $Y^*$. If there is a $2$-cell attached along some loop in ${Y'}^*$, there must be finitely many $2$-cells attached along all its preimages in $Y^*$. Therefore $X^*$ is also a finite cover of ${X'}^*$. $G=\pi_1({X'}^*)$, $H=\pi_1(X^*)$ and so $[G:H]$ is finite. 
\end{proof}

In \cite{rey}, Reynold proves that for a given $\F_n$-tree $T$ with \textit{very small} action, suppose that $T$ is \textit{indecomposable} and  $H<\F_n$ is finitely generated, the action $H$ on its minimal invariant subtree $T_H$ has dense orbits if and only if $H$ has finite index in $\F_n$. Notice that a band complex consisting a single surface or thin component is dual to a indecomposable (geometric) tree. Thus Corollary~\ref{output2} is a weaker version of Reynolds's result for all finitely presented groups.

\subsection{General case}\label{generalcase}
   In previous sections, for a given pair of band complexes $(X \overset{\i}{\to}X')$, we always assume $(A5)$: $\i$ is injective between the sets of components in $Y$ and in $Y'$, to simplify our description. In fact, the assumption $(A5)$ is not necessary for the relative Rips machine. Suppose that there are several components $Y_0^1, \dots, Y_0^k\subset Y$ map to the same component $Y'_0\subset Y'$. The relative Rips machine described in section \ref{process} for a pair can also be applied to a $k+1$-tuple $(Y_0^1,\dots, Y_0^k, Y'_0)$. Roughly speaking, in each step, we may apply proper moves to $Y'_0$ first, then modify each $Y_0^i$ separately. 
   
   In more detail, if $Y'_0$ contains some free subarc, Process I will be applied. First in $Y'_0$, collapse from a free subarc $J_{Y'}$ to obtain $Y'_1$. Similarly as in Section~\ref{processI}, let $J_1^j, \dots, J_{n_j}^j$ be pre-images of $J_{Y'}$ in $Y_0^j$. Then collapse from $J_1^j, \dots, J_{n_j}^j$ in $Y_0^j$ for $j=1, \dots, k$ to obtain $(Y_1^1,\dots, Y_1^k, Y'_1)$. 
   
   If $Y'_0$ contains no free subarc, Process II will be applied. Considering each $(Y_0^j\overset{\i_0^j}{\to}Y'_0)_j$ as a pair, $Y_1^j$ is obtained by modifying $Y_0^j$ as described in Section~ \ref{processII}. A sequence $(Y_0^1,\dots, Y_0^k, Y'_0),(Y_1^1,\dots, Y_1^k, Y'_1),\dots $ produced by Process I and Process II will eventually \textit {stabilize} as proved in Proposition ~\ref{mimicgraph}. Now for a stabilized $k+1$-tuple  $( Y_i^1,\dots, Y_i^k, Y'_i)$, if none of $Y_i^j$ contains pre-weight $0$ bands or pre-weight $\frac{1}{2}$ bands, continue with Process I. Otherwise, apply Process III to $(Y_i^1,\dots, Y_i^k, Y'_i)$. $(Y_{i+1}^1,\dots, Y_{i+1}^k, Y'_{i+1})$ can be obtained as in Section ~\ref{processIII} by considering each $(Y_i^j\overset{\i_i^j}{\to}Y'_i)_j$ as a pair. Then by Proposition ~\ref{finalrelation}, Process III appears only finitely many times. 
   
  Following exactly the same arguments, all results proved in Section~ \ref{special} hold for pairs of band complexes without assuming $(A5)$.
      
 

\section{A partial order on band complexes}\label{s:app}   
We define a partial order on band complexes in this section. 

\begin{definition}\label{d:orderofbc}

A morphism $\i: X\to X'$ induces a cellular map $\phi: \Delta(X)\to \Delta(X')$ between their $GD$'s. We say that a vertex of $\Delta(X)$ \textbf{changes type under} the map $\phi$, if the vertex and its image have different types. For example, a simplicial component maps into a minimal component.

For a vertex $v$ of $\Delta(X)$, let its stabilizer $\pi_1(v)$ be the conjugacy class of the subgroup of $\pi_1(X)$ generated by the union of bands contained in $v$. We say $v$ \textbf{gets bigger} if the stabilizer of the image of $v$ strictly contains the image of the stabilizer of $v$, i.e. $\phi(\pi_1(v))\subset \pi_1(\phi(v))$. For example, a surface component maps to another surface component as a 2-sheeted cover.
\end{definition}

\begin{definition}\label{complexity}
   For band complexes $X$ and $X'$, we write $X<X'$ if there is a morphism $\i: X\to X'$ such that: 
   \begin{enumerate}
              \item $\Delta(X)$ has more vertices than $\Delta(X')$; or
             \item $\Delta(X)$ and $\Delta(X')$ have the same number of vertices and under the morphism either a vertex changes type or a surface or thin vertex gets bigger.
                  
   \end{enumerate}
\end{definition}

\begin{prop}\label{FSC}
   Every sequence of band complexes, $X_1<X_2<\dots$, eventually stabilizes. 
\end{prop}

\begin{proof}

Let $Y_i$ be the underlying union of bands for $X_i$  and $\i_i: X_i\to X_{i+1}$ be a morphism such that $X_i<X_{i+1}$ holds.  It is enough to consider a special sequence $Y_0<Y_1< Y_2< \dots$ with the property that each $Y_i$ is a single component of surface or thin type. $(Y_i \overset{\i_i}{\to} Y_{i+1})$ is then a pair of components. According to Theorem~ \ref{t:main}, up to omitting weight $0$ bands,  the relative Rips machine eventually converts $Y_i$ into a finite cover of $Y_{i+1}$. In the case of surface components, the length of a sequence of finite covers is bounded by the absolute value of the Euler characteristic of $Y_1$. In the case of thin components, associated to $Y_i$, there is a graph according to Section ~\ref{s:thinstructure}: consider each island as a vertex and each long band as an edge. The length of a sequence is then bounded by the absolute value of the Euler characteristic of the graph associated to $Y_1$. 
\end{proof}

\begin{remark}
Let $(X \overset{\iota} \to X')$ be a pair of band complexes. Assume that the preimage of each component $Y'_{0}$ of $Y'$ is not empty. Then either $X<X'$; Or $\i^{-1}(Y'_{0})$ is a single component, denoted by $Y_{0}$. Moreover, if $Y'_{0}$ is either a surface or thin component, then $(Y_{0} \overset{\iota} \to Y'_{0})$ further has the property that the relative Rips machine eventually converts it into a pair $(Y^*_{0} \overset{\iota^*} \to {Y'_{0}}^*)$ where $\overline{Y_{0}^*}$ is homeomorphic to $\overline{{Y'_{0}}^*}$.

In the case of $X<X'$, we may build a finite collection $\{X^*\}$ of band complexes following the instruction below such that either $X^*=X'$ or $X^*$ is an ``intermediate'' band complex between $X$ and $X'$, i.e. $X<X^*<X'$.

Here is a brief description on how to construct $X^*$'s. Such a $X^*$ is called an \textit{enlargement} of $X_H$. More detail of enlargements will be discussed in \cite{definable2}.

If $\i^{-1}(Y'_{0})$ contains two or more components, $X^*$ is obtained by gluing two of these components together. Let $Y^1_{0}$ and $Y^2_{0}$ be two distinct components that map into $Y'_{0}$ and let $J$ be an arc in the underlying real graph of $Y'_{0}$ with a positive transverse measure. By picking $J$ small enough, we may assume that $J$ has finitely many pre-images $J^1_1\dots J^{1}_{n_1}$ in $Y^1_{0}$ and finitely many pre-images $J^2_1\dots J^2_{n_2}$ in $Y^2_{0}$ with the property that $\i$ maps each $J^i_j$ isometrically onto $J$. We may then build $X^*$ by attaching a band $B=J\times I$ to $Y$ via a measure preserving map $J\times \{0\}\to J^1_m$ and $J\times \{1\}\to J^2_{m'}$ for some $m\in \{1,\dots, n_1\}, m'\in \{1,\dots, n_2\}$ \footnote{We may add one band for each pair of segments in $\{J^1_1\dots J^{n_1}_1, J^1_2\dots J^{n_2}_2\}$ to get a bigger intermediate band complex.} (add extra $2$-cells if necessary).  

If $Y'_0$ is a surface or thin component and $\i^{-1}(Y'_0)$ is a single component $Y_0$ which is either simplicial or a proper cover of $Y'_0$ (omitting weight $0$ bands), $X^*$ is obtained from $X$ by ``replacing'' $Y_0$ with $Y'_0$. Let $J$ be an small arc of the underlying real graph contained in $\i(Y_0)$ and $J'$ be one of $J$'s preimages in $Y_0$. Since $Y'_0$ is either a surface or thin component, it has a band complex structure over $J$, i.e. $Y'_0$ can be viewed as a union of finitely many bands glued to $J$. Now we may glue a copy of $Y'_0$ to $Y_0$ by identifying $J$ and $J'$. Then adding extra $2$-cells if necessary, the resulting complex is $X^*$. 

\end{remark}

\appendix

\section{Thin Type components} \label{thin}

\subsection{Review.}\label{thinreview} \hspace*{\fill} 

Fix $X$ be a band complex with underlying union of bands $Y$. Let $Y_0\subset Y$ be a component. As described in Section~\ref{background}, $Y_0$ is of thin type if eventually only Process I is applied. In this section, we will review Process I and restate some properties proved in \cite{bfstableactions}  for our convenience. 
 
\begin{definition}\label{generalizedband}
Let $B_1, B_2, \dots, B_n$ be a sequence of weight 1 bands in $Y$. We say that they form a \textbf{generalized band} $\BB$ provided that :
   \begin{itemize}
      \item the top of $B_{i}$ is identified with the bottom of $B_{i+1}$ and meets no other positive weight bands for $i=1, 2, \dots, n-1$, and
      \item the sequence of bands is maximal with respect to above property.
   \end{itemize}
$B_i$ and $B_{i+1}$ are called \textit{consecutive bands}. The bottom of $B_1$ and the top of $B_n$ are \textit{bases} of $\BB$. 
For a generalized band $\BB=b\times I_n$ where $I_n=[0,n]$, a \textit{vertical fiber} is a set of the form $\{point\}\times I_n$. 
A sequence of consecutive bands $B_i, B_{i+1},\dots, B_{i+k}$ is called a \textit{section} of $\BB$, where $i\geq 1, i+k\leq n$. The \textit{length} of $\BB$, denoted by $l(\BB)$, is $n$ (the number of bands it contains). The \textit{width} of $\BB$, denoted by $w(\BB)$, is the transverse measure of its base $b$. Let $c\subset b$ be a sub-interval, $\CC=c\times I_n$ is a \textit{sub-generalized band} of $\BB$. Similarly as for bands, we may talk about weight of a generalized band.
\end{definition}   

\begin{definition}
For a given band, the midpoint of a base divides the base into \textbf{halves}. A weight 1 base is \textbf{isolated} if its interior does not meet any other positive weight base. A half $h$ of a weight $\frac{1}{2}$ base $b$ is \textbf{isolated} if the interior $h$ meets no positive weight base other than $b$ and $\text{dual}(b)$. A weight $1$ base $b$ is \textbf{semi-isolated} if a deleted neighborhood in $b$ of one of its endpoints meets no other positive weight base. A half $h$ of a weight $\frac{1}{2}$ base $b$ is \textbf{semi-isolated} if a deleted neighborhood in $h$ of one of its endpoints meets no positive weight base other than $b$ and $\text{dual}(b)$.
\end{definition}   

\vspace{1pc}
We are now ready to describe Process I of the Rips machine. 
\vspace{1pc} 

\noindent\textit{Process I}. Fix a minimal component $Y_0$ in $X$. We define $X'$ to be the band complex obtained from $X$ by the following operation. Find, if possible, a maximal free subarc $J$ of a base $b$ in $Y_0$. If such a $J$ does not exist, define $X'=X$ and go on to Process~II. Now use $(M5)$ to collapse from $J$. If there are several $J$'s to choose from, abide by the following rules:
   
   \begin{enumerate}
      \item If there is an isolated (half-) base $c$, set $J=c$. This is called an $I_1$-collapse.
      \item If there is no isolated (half-) base, but there is a semi-isolated (half-) base $c$ then choose $J$ so that it contains an endpoint of $c$.  This is called an $I_2$-collapse.
      \item If there are no isolated or semi-isolated (half-bases), we can use any free subarc as J. This is called an $I_3$-collapse.
      \item Generalized bands are treated as units. More precisely, if $J$ is contained in a base of a generalized band, the collapse is preformed all the way along the entire generalized band.
         \end{enumerate}

\vspace{1pc}
\begin{notation}\label{sequencenotation}
Let $X=X^0, X^1,\dots$ be an infinite sequence of band complexes formed by successive applications of Process I, i.e. $X^{i+1}=(X^i)'$. Further let $Y^i$ be the underlying union of bands for $X^i$ and $Y_i\subset Y^i$ be the component corresponding to $Y_0$. Denote the collapsing map from $Y_i$ to $Y_{i+1}$ by $\delta_i$.
Omitting weight $0$ bands, there is a  natural inclusion $\overline{Y}_{i+1}\hookrightarrow \overline{Y}_{i}$, denoted by $\i_i$.

In $Y_i$, let the number of generalized bands that have positive weight be $a_i$. We say a generalized band $\BB_{i+1}\subset \overline{Y}_{i+1}$ is an \textbf{image} of a generalized band $\BB_{i}\subset \overline{Y}_{i}$, if a \textit{section} of $\BB_{i+1}$ is contained in $\BB_{i}$ as a sub-generalized band under $\i_i$. In particular, after collapsing a free subarc in $\BB_{i}$, $\BB_{i}$ has no image if it is an $I_1$-collapse, has exactly one image if it is an $I_2$-collapse and has at most two images if it is an $I_3$-collapse. It is clear that in all cases $l(\BB_{i})\leq l(\BB_{i+1})$ when $\BB_{i+1}$ is an image.  A collapse $\delta_i$ is said to be \textbf{increasing} if the length of some generalized band in $Y_i$ is strictly less than the length of one of its images in $Y_{i+1}$.
\end{notation}

\begin{lemma}\label {increasing}
Follow the above notation. For a given collapse $\delta_i: Y_i\to Y_{i+1}$,  we have:
   \begin{itemize}
      \item If $\delta_i$ is an $I_1$-collapse, then $a_{i+1}=a_i-1$ and $\delta_i$ is not increasing.
      \item If $\delta_i$ is an $I_2$-collapse, either $a_{i+1}=a_i$ or $a_{i+1}=a_i-1$. Moreover, $\delta_i$ is an increasing collapse if and only if $a_{i+1}=a_i-1$.
      \item If $\delta_i$ is an $I_3$-collapse, then $a_{i+1}=a_i+1$ or $a_{i+1}=a_i$ or $a_{i+1}=a_i-1$. Moreover, $\delta_i$ is an increasing collapse if and only if $a_{i+1}\neq a_i+1$. 
   \end{itemize}

\end{lemma}

\begin{proof}
The first item is clear since a whole generalized band vanishes in an $I_1$-collapse. 

For the second item, let $\CC_i$ be a generalized band in $Y_i$. Suppose that one of its bases $c_i$ is a semi-isolated base and $\delta_i$ collapses from a maximal free subarc $J\subset c_i$. In $Y_{i+1}$, let the remaining part of $\CC_i$ be $\CC'_i$ and the image of $\CC_i$ \footnote{$\CC_i$ has exactly one image since $\delta_i$ is an $I_2$-collapse.} be $\CC_{i+1}$ (i.e. the generalized band containing $\CC'_i$), see Figure~\ref{Ft1}. If $\CC'_{i}=\CC_{i+1}$, then $a_{i+1}=a_i$ and $\delta_i$ is not increasing. If $\CC'_i\subsetneq \CC_{i+1}$, then $\CC_{i+1}$ must be the concatenation of $\CC'_{i}$ and another generalized band (can not concatenate with more than one band due to no proper compact subset of each leaf).
In this case, $a_{i+1}=a_i-1$ and $\delta_i$ is increasing. The third item can be argued in the same fashion. 

\end{proof}
 
 \begin{figure}[h!]
   \centering
      \begin{tikzpicture}[thick, scale=0.8]
      \footnotesize
\draw (0,0)--(2,0)--(2,2)--(0,2)--(0,0);
\draw (-0.5,3)--(0,2)--(1.5, 2)--(1,3);
\draw (2,3)--(1.5,2)--(2, 2)--(2.5,3);
\draw (1,-1.5)--(1,0)--(2, 0)--(2,-1.5);
\draw [ultra thick] (0,0)--(1,0);
\node [below] at (0.5,0) {$J$};

\draw [dashed] [red] (2,0)--(2.5,0)--(2.5,-1.5);
\node [left] at (0,0) {$c_i$};
\node [above] at (1, 0.7) {$\CC_i$};
\node [left] at (0, 2) {$dual(c_i)$};

\node [above] at (1.5, -1) {$\mathbf{A}_i$};
\node [below] at (1, -2) {$Y_i$};

\draw [->] (4,0)--(6,0);
\node [above]  at (5,0) {Collapse $J$};

\draw (9,-1.5)--(9,2)--(8,2)--(8,-1.5);
\draw (6.5,3)--(7,2)--(8.5, 2)--(8,3);
\draw (9,3)--(8.5,2)--(9, 2)--(9.5,3);
\draw [dashed] [red] (9,0)--(9.5,0)--(9.5,-1.5);
\draw [dashed] (8,0)--(9,0);
\node [above] at (8.5, 0.7) {$\CC'_i$};
\node [left] at (8, -.03) {$\CC_{i+1}$};
\node [above] at (8.5, -1) {$\mathbf{A}'_i$};
\node [below] at (8, -2) {$Y_{i+1}$};
\end{tikzpicture}
\caption{\footnotesize Without the red dashed lines, the picture shows the case that $\CC'_i\subsetneq \CC_{i+1}$ and $a_{i+1}=a_i-1$. In particular, $\CC'_i$ and $\mathbf{A}'_i$ form a longer generalized band $\CC_{i+1}$ in $Y_{i+1}$. If $\mathbf{A}_i$ has the red dashes as its boundary, then $\CC'_i=\CC_{i+1}$ and $a_{i+1}=a_i$.} \label{Ft1}
\end{figure}
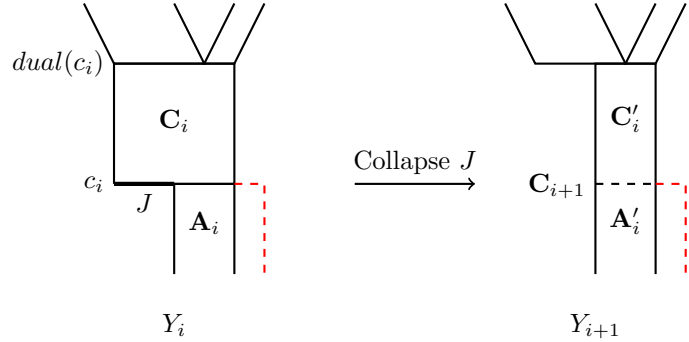

   \begin{prop}\label{thinprop} 
   Let $X$ be a band complex with underlying union of bands $Y$ and $Y_0\subset Y$ be a minimal component of thin type.
   Suppose that $Y_0, Y_1,\dots$ is an infinite sequence of components formed by application of Process I. Then the following holds.
      
      \begin{enumerate}
         \item $\textit{Complexity}(Y_{i+1})\leq \textit{Complexity}(Y_i)$ for $i=0, 1, 2\dots$. In particular, eventually the complexity is a fixed number $\mathcal{C}$.

         \item All bands in $\overline{Y}_i$ can be organized into at most $N_B=6\mathcal{C}+1$ generalized bands for all sufficiently large $i$.
         
         \item Let $L$ be the induced lamination on $\overline{Y}$ and the limiting lamination be $L_{\infty}=L\cap \bigcap_{i}^{\infty} \overline{Y}_i$. Then each subband of a band in $\overline{Y}$ either fully collapses within finitely many steps or meets $L_{\infty}$ in infinitely many vertical fibers.
         
         
      \end{enumerate}

   \end{prop}

\begin{proof}  
(1) This follows from the definition of complexity.
Suppose that $\delta_i: Y_{i}\to Y_{i+1}$ is an $I_3$-collapse. Let the collapsed free subarc be $J_i$, the block containing $J_i$ be $\b_i$ and the block containing $\text{dual}(J_i)$ be $\b^*_i$. 
$\b_i$ is necessarily split into two blocks $\b_{i+1}^1,\b_{i+1}^2 $ contained $ \Gamma_{i+1}$. Denote the remaining of $\b^*_i$ in $\Gamma_{i+1}$ by $\b^*_{i+1}$, which may be a single block or a union of two blocks. Then number of bases contained in the union $\b_{i+1}^1\cup \b_{i+1}^2\cup \b^*_{i+1} $ is at most two more than the number of bases contained in $\b_i\cup \b^*_{i}$. Therefore: $$\text{Complexity}(\b_{i+1}^1)+\text{Complexity}(\b_{i+1}^2)+\text{Complexity}(\b^*_{i+1})$$
$$\leq \text{Complexity}(\b_i)+\text{Complexity}(\b_{i}^*)$$
Thus $\text{Complexity}(Y_{i+1})\leq \text{Complexity}(Y_i)$. Similar analysis can be done for both $I_1$- and $I_2$-collapse.

          (2) According to the first item, without loss, we may assume that $Y_i$ has the fixed complexity $\mathcal{C}$. Therefore, in $Y_i$, the number of blocks that have positive complexity is bounded by $\mathcal{C}$.
         Recall that the number of generalized bands in $\overline{Y}_i$ is denoted by $a_i$. Suppose that $a_{i+1}>a_i$. This occurs only if $\delta_i: Y_i\to Y_{i+1}$ is an $I_3$-collapse (Lemma~\ref{increasing}). Thus $Y_i$ contains no isolated or semi-isolated bases. Therefore, each block of complexity 0 consists of two coinciding bases of weight $1$, or of a pair of weight $\frac{1}{2}$ bases that coincides with a base of weight 1. So for a given generalized band $\BB$,  either (at least) one of its bases is contained in a block of positive complexity or $\BB$ has weight $\frac{1}{2}$ and its bases form a block of complexity 0 along with another weight 1 base. It follows that $a_i$ is bounded above by $6\mathcal{C}$, and so $a_{i+1}\leq 6\mathcal{C}+1$. 

          (3) Assume that $B\subset \overline{Y}$ is a subband that does not fully collapse within finitely many steps. According to \cite[Proposition 7.2]{bfstableactions}, $I_3$-collapses occur infinitely often along the application of the Rips machine. So the number of connect components of the intersection between $B$ and $\overline{Y}_i$ goes to infinity with $i$. Each of these components then have a nonempty intersection with $L_{\infty}$ and we are done. 
          \end{proof}
          
          
          
          

   \vspace{1pc} 
Let $X$ be a band complex with underlying union of bands $Y$. For a fixed thin component $Y_0\subset X$,  let $X=X^0, X^1, \dots $ be an infinite sequence of band complexes formed by application of the Rips machine with respect to $Y_0$. In general, it is possible that the process of the Rips machine that is applied to $X^i$ to obtain $X^{i+1}$, bounces between Process I and Process II for a while before it eventually stabilizes with Process I. Moreover, the complexity of $X^i$ may actually decrease at the beginning stages. Nonetheless, there exists some integer $A$ such none of these above situations happens after $X^A$. So without loss, for the rest of Section ~\ref{thin}, we always make the assumption that in such a sequence $X=X^0, X^1, \dots $, only Process I occurs and the complexity is fixed for all $X^i$'s.  

Let $B\subset Y$ be a band of positive weight. We say $B$ \textit{vanishes} along the process if $B\cap Y^i=\emptyset$ for sufficiently large $i$ ($\overline{Y}^i \hookrightarrow \overline{Y}^0$). It is clear that bands only vanish in $I_1$-collapses. Since $I_1$-collapses reduce the complexity, under our fixed complexity assumption, we may further assume that only $I_2$-collapses and $I_3$-collapses occur along a sequence $X^0, X^1, \dots$. In particular, no band vanishes.

\subsection {Structure of Thin Type.}\hspace*{\fill}\label{s:thinstructure}

We will continue with the same notation as in Section~\ref{thinreview}. 
Denote the union of bases of all generalized bands in $\overline{Y}_i$ by $E_i$. Then $E_i$ is a union of at most $2N_B$ closed intervals where $N_B$ is the uniform upper bound of the number of generalized bands as in Proposition~\ref{thinprop}. $\overline{Y}_{i+1}\subset \overline{Y}_i$ implies that the sequence $E_1\supset E_2\supset \dots $ is nested. According to  \cite[proposition 8.12]{bfstableactions}, $\max\{\text{widths of generalized bands of } Y_i\} \to 0$ as $i\to \infty$. Therefore, the intersection $\displaystyle \cap_{i=0}^{\infty} E_i$ consists of at most $2N_B$ points. 
Denote these points by $e_1,e_2,\dots, e_n$ where $n\leq 2N_B$. Note that if a vertical fiber of a generalized band in $Y_i$ contains some $e_j$, $e_j$ must be an endpoint of that fiber.

A vertical fiber of a generalized band in $\overline{Y_i}$ is \textbf{short} if both its endpoints are contained in $\{e_1,e_2,\dots, e_n\}$. 

Let $s_1, s_2, \dots, s_m$ be the list of \textit{short} vertical fibers. Since widths of bands in $Y_i$ converge to $0$ with $i$, each generalized band in $Y_i$ contains at most one $s_j$ for $i$ sufficiently large. In particular, $m\leq N_B$. 
By definition of short fibers, we have the following.

\begin{lemma}\label{si}
In $\overline{Y}_i$, every loop in a leaf is contained in $\displaystyle \mathcal{S}=\cup_{j=1}^{m} s_i$. 
\end{lemma}

\begin{definition}\label{short}
    In $Y_i$, we say a generalized band of positive weight is \textbf{short} if one of its vertical fibers is a subset of a \textit{short} vertical fiber. Otherwise, we say a generalized band is \textbf{long}. In particular, for $i$ sufficiently large, a generalized band is short if and only if it contains one of $\{s_j\}_j$.
\end{definition}

\begin{prop} \label{thin structure}
Following the same notation as in Proposition~\ref{thinprop}, we have that in $Y_i$'s: 
   \begin{enumerate}
      \item Lengths of short bands are bounded above by $L_s=\max_{j\in \{1,\dots, m\}}\{l(s_j)\}$.
      \item There exists an integer $M_s$ such that for any $i>M_s$, $Y_i$ contains exactly $m$ short bands. Further, for all $i>M_s$, there is a natural bijection between short bands in $Y_i$ and $Y_{i+1}$. More precisely, every short band $\BB_{Y_{i+1}}$ in $Y_{i+1}$ is a sub-generalized band of a short band $\BB_{Y_i}$ in $Y_i$ (induced by the inclusion $\overline{Y}_{i+1}\hookrightarrow \overline{Y}_i$).
      \item $Y_i$ contains at least one long band, for all large $i$.
      \item For any fixed number $k$, there exists a number $N(k)$ such that the lengths of long bands in $Y_i$ are all greater than $k$ for $i>N(k)$. In particular, lengths of long bands are going to $\infty$ as $i\to \infty$ in $Y_i$.

   \end{enumerate}

\end{prop}

\begin{proof}
$(1)$, $(2)$ follow directly from Definition~\ref{short}.

(3) Since the number of generalized bands in $Y_i$ is uniformly bounded, we only need to show that there are infinitely many \textit{increasing} $\delta_i$'s in the the sequence $Y_0\overset{\delta_0} \to Y_1 \overset {\delta_1}\to Y_2 \dots$. We are going to construct an infinite subsequence of $\{Y_i\}_i$ such that for each $Y_i$ in that subsequence, its corresponding $\delta_i$ is increasing. Firstly, since $a_i$, the number of generalized bands in $Y_i$, is bounded by $N_s$. There exists a subsequence  $A=\{Y_{k_1}, Y_{k_2},\dots\}$ consists all $Y_{k_j}$'s with the property that for any fixed $j$,  $a_{k_j}\geq a_i$ for all $i\geq k_j$. Starting with $Y_{k_1}$, if $\delta_{k_1}$ is increasing, we continue to check $\delta_{k_{2}}$. Otherwise, by Lemma ~\ref{increasing}, either $\delta_{k_1}$ is an $I_2$-collapse with $a_{k_1+1}=a_{k_1}$, in which case we delete $Y_{k_1}$ from $A$ ($Y_{k_1+1}$ must also in $A$); Or $\delta_{k_1}$ is an $I_3$-collapse with $a_{k_1+1}>a_{k_1}$, which contradicts to the choice of $Y_{k_1}$. Since $I_3$-collapses happen infinitely often in $\{Y_0, Y_1,\dots\}$ (\cite[Proposition 7.2]{bfstableactions}), $A$ is an infinite set. 

(4) Let $l_i$ be the length of a shortest long bands in $Y_i$. Then the sequence $\{l_i\}_i$ is non-decreasing as $i\to \infty$. We need to prove that the sequence is indeed increasing. Assume that it is bounded above. Then for all sufficiently large $i$, $l_i$ is a fixed number $l^*$. Note that each long band in $Y_{i+1}$ with length $l^*$ is a sub-generalized band of some long band in $Y_i$, which also has length $l^*$. This implies that there exists a sequence of long bands $\{\BB_{Y_i}\}_i$ such that $\BB_{Y_{i+1}}\hookrightarrow \BB_{Y_i}$ and $l(\BB_{Y_i})=l^*$ for all $i$'s. Thus $\displaystyle \cap_{i}^{\infty} \BB_{Y_i}\in \{s_1, s_2,\dots, s_m\}$. But this contradicts the assumption that $\BB_{Y_i}$'s are long bands.\end{proof}

According to the above proposition, a generalized band in $Y_i$ is short if and only if its length is less than $L_s$, for $i>\max\{M_s, N(L_s)\}$. Otherwise it is long (thus the name).

\begin{definition}\label{island}
   $\{e_1,\dots, e_n\}$ are defined as above. In $Y_i$, an \textbf{island} is either 
   \begin{itemize}
       \item \textit{Type 1} a connected component of the union of short bands along with the blocks it intersects, see Figure~\ref{fgisland},
       or
       \item \textit{Type 2} a block containing an $e_j$ which is not contained in an type 1 island.
   \end{itemize}

   An island is \textbf{contractible} if it is homotopy equivalent to a point. In particular, every type 2 island is contractible.
\end{definition}

\begin{figure}[h!]
\centerline{\includegraphics[width=0.4\textwidth]{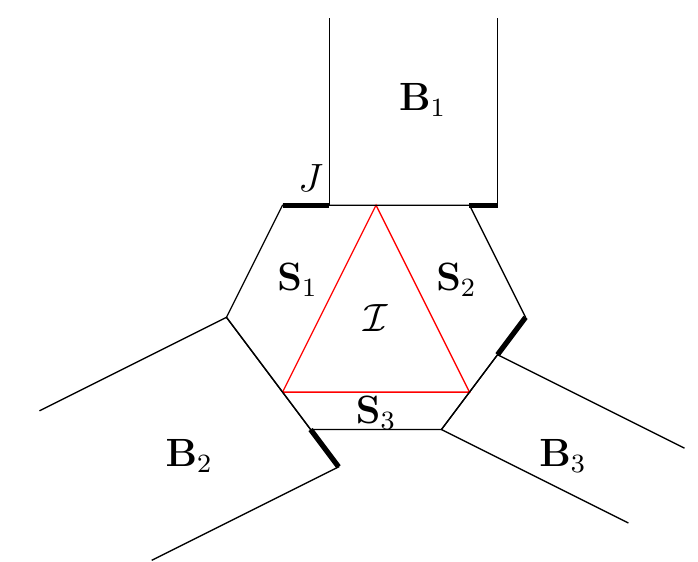}}
\caption{\footnotesize A type 1 island $\mathcal{I}$ formed by the union of short bands $\mathbf{S}_1,\mathbf{S}_2,\mathbf{S}_3$ (determined by the union of \textit{short fibers} showing in red). $\BB_1,\BB_2$ and $\BB_3$ are long bands, one of whose bases is contained $\mathcal{I}$. Suppose that $Y_{n+1}$ is obtained from $Y_n$ by collapsing from $J$. This preserves the homeomprhism type of $\mathcal{I}$.} \label{fgisland}
\end{figure}

\begin{prop}\label{islandprop}
    Let $X$ be a band complex with underlying union of bands $Y$ and $Y_0\subset Y$ be a component of thin type.
   Suppose that $Y_0, Y_1,\dots$ is an infinite sequence formed by application of Process I.
   
   \begin{enumerate}
   \item There exists an integer $M_I$ such that the number of islands in $Y_i$ is fixed for all $i>M_I$. More precisely,  let $\displaystyle \mathcal{S}=\cup_{j=1}^{m} s_i$, $\displaystyle \mathcal{E}=\cup_{j=1}^{n} e_i$ and $N_{I}$ be the fixed number for islands. $N_{I}$ equals the number of connected components of $\mathcal{S}\cup\mathcal{E}$. 
   
   \item For $i>M_I$, there exists a natural bijection $f_i$ between the set of islands in $Y_i$ and the set of islands in $Y_{i+1}$ with the following properties:
   \begin{enumerate}
   \item The restriction of $f_i$ to each island is a homotopy equivalence.
   \item The restriction of $f_i$ to $\mathcal{S}\cup\mathcal{E}$ is the identity map.
   \item $f_i$ maps vertical fibers to vertical fibers
   \item After passing to a subsequence, $f_i$ is a homeomorphism on each island.
   \end{enumerate}
   \end{enumerate}  
   
       
      
      
     

\end{prop}

\begin{proof}
(1) To show the existence of $M_I$, we only need to verify two facts. One is that eventually distinct short bands either stay connected or disconnected. The second is that each $e_j$ is either always or never contained in a type 1 island. For $i$ large, let $\BB^1_{Y_i}$ and $\BB^2_{Y_i}$ be two short bands in $Y_i$. Without loss, we may assume that $s_1\subset \BB^1_{Y_i}$, $s_2\subset \BB^2_{Y_i}$ are short vertical fibers determining them. In $Y_{i+j}$, denote the short band containing $s_k$ by $\BB^k_{Y_{i+j}}$, $k=1,2$. If $s_1$ and $s_2$ share an endpoint, it is clear that $\BB^1_{Y_{i+j}}$ and $\BB^2_{Y_{i+j}}$ are connected for any $j\geq 0$. Otherwise, there exists an $\epsilon >0$ such that $\epsilon$-neighborhoods of endpoints of $s_1$ and $s_2$ are disconnected. Since widths of generalized bands in $Y_i$ converge to $0$ as $i\to \infty$. There exists $M^*>0$ such that $\BB^1_{Y_{i+j}}$ and $\BB^2_{Y_{i+j}}$ are disconnected for all $i+j>M^*$. Similar argument shows that there exists $M^{**}$ such that every $e_j\notin \mathcal{S}$ is not contained in any type 1 island. Thus for all $i>M_I=\max\{M^*, M^{**}\}$, every component of $\mathcal{S}\cup\mathcal{E}$ determines a unique island in $Y_i$. So $N_{I}$ equals the number of components of $\mathcal{S}\cup\mathcal{E}$.

(2) In $Y_i$, let the island determined by a connected component $\alpha$ of $\mathcal{S}\cup\mathcal{E}$ be $\mathcal{I}^{\alpha}_i$. 
Let $f_i$ be the function that takes an island $\mathcal{I}^{\alpha}_i$ in $Y_i$ to its corresponding island $\mathcal{I}^{\alpha}_{i+1}$ in $Y_{i+1}$. It is easy to check that $f_i$ is a bijection. 

Now consider the restriction to a fixed island of $f_i: \mathcal{I}^{\alpha}_i \to \mathcal{I}^{\alpha}_{i+1}$. If $\mathcal{I}^{\alpha}_i$ is a type 2 island, (a)-(d) hold for any map that is the identity on $e_j\in \alpha$. If $\mathcal{I}^{\alpha}_i$ is a type 1 island. Both $\mathcal{I}^{\alpha}_i$ and $\mathcal{I}^{\alpha}_{i+1}$ deformation retract to $\alpha$, thus (a) and (b). In more detail, for each short band $\BB_{i}=b_i\times I_n\in \mathcal{I}^{\alpha}_i$ and its corresponding short band  $\BB_{i+1}=b_{i+1}\times I_n\in \mathcal{I}^{\alpha}_{i+1}$, let $f_i|_{\BB_{i}}=h\times id_{I_n}$ where $h: b_i\to b_{i+1}$ is a surjective map, so (c) holds. Moreover, the combinatorial type of an island with fixed number of bands is bounded. After passing to a sub-sequence, (d) holds.
\end{proof}

\begin{definition}\label{naked}
   A generalized band $\BB=b\times [0,n]$ of a band complex $X$ is \textbf{naked} if the only cells of $X$ meet $b\times (0,n)$ are subdivision annuli. A generalized band $\BB$ is \textbf{very naked} if $b\times (0,n)$ meets no cell of $X$.
\end{definition}

By the definition of $\{e_j\}$, every block of $Y_i$ contains a point in $\{e_j\}$ for $i$ sufficiently large. Therefore, each base of a long band is contained in a block and the intersection between this block and some island is non-degenerate (not a point). Thus, every long band can be viewed as traveling from one island to another island (maybe the same one). Moreover, it is showed in \cite[proposition 8.13]{bfstableactions} that every long band is a naked band. Considering each island as a vertex, each long band as an edge and each weight $0$ band attached to a long band as an edge group, a thin component then has  a ``graph of spaces'' like structure. 


\vspace{1pc}

In fact, after passing to a subsequence, $f_i$ constructed in Proposition ~\ref{islandprop} can be extended into a homeomorphism from $Y_i$ onto $Y_{i+1}$. It is because that the combinatorial type of a $Y_i$ as a union of bands is bounded. More generally, $f_i$ can be further extended to entire band complexes as follows.

Let $\{Y_i\}_i$ be a subsequence such that $\{f_i: Y_i\to Y_{i+1}\}_i$ are homeomorphisms. In $Y_0$, let $u_0$ be a non-degenerate horizontal segment. Inductively define that $u_{i+1}=f_i(u_i)\subset Y_{i+1}$. Recall that the band complex containing $Y_i$ is denoted by $X^i$

\begin{cor}\label{thinboundary}
Suppose that every contractible attaching region of $Y_i$ is contained in $u_i$. Then after passing to a further subsequence, $f_i$ can be extended into a homeomorphism $f^i: X^i\to X^{i+1}$ with the property that $f^i$ is the identity map on the closure of the complement of $Y_i$ in $X^i$, for all $i$.
\end{cor}

  
\begin{proof}
Denote the closure of complement of $Y_i$ in $X^i$ by $Z^i$. Note that $Z^i$ is the same space for all $i$. Let $f^i: Z_i\to Z_{i+1}$ be the identity map. We only need to check that $f_i$ is well-defined on attaching regions (intersections between $Y_i$ and $Z^i$). Since every attaching region is contained in a leaf, if an attaching region is not contractible, it must be contained in $\mathcal{S}$. Therefore, there $f_i: Y_i\to Y_{i+1}$ and $f^i: Z_i\to Z_{i+1}$ are compatible on non-contractible attaching regions (Proposition~\ref{islandprop} bullet (b)). Now, in total, there are finitely many contractible attaching regions, each of which is a point in $u_i$. By the assumption, after passing to a further subsequence, we may pick $f_i$ such that it maps each contractible attaching region, say $p$ to its image under $f^i$, i.e. $f_i(p)=f^i(p)$. Thus, we are done. \end{proof}


\begin{remark}\label{goodattachingregion}
The assumption that every contractible attaching region in $Y_i$ is contained in $u_i$ in Corollary~\ref{thinboundary} can be arranged for any sequence of band complexes $\{X^i\}_i$ as follows. For a fixed contractible attaching region (a point), suppose that a point $q\in Z_i$ is attached to a point $p\in Y_i$. Let $l_{p}\subset Y_i$ be the leaf containing $p$ and $p'$ be a point in $\in u_i\cap l_{p}$ ($l_{p}$ intersects $u_i$ infinitely often). We can then slide $q$ across $l_p$ to $p'$.  As there are finitely many attaching regions, the assumption holds.  \end{remark}

\subsection{Shortening Thin Type}\label{shortenthin}\hspace*{\fill} 

For a band complex $X$ with underlying union of bands $Y$, let $m$ be the transverse measure on $X$ and $\B$ be a set of loops in $X$ generating $\pi_1(X)$. Set $|X|_{\B}:=\sum_{\mu\in \B}m(\mu)$ be the \textbf{length} of $X$ with respect to $\B$. 
If $a: X\to X$ is a morphism that is also a homotopy equivalence with the property $|X|_{a(\B)}<|X|_{\B}$, then we say that $X$ (w.r.t. $\B$) is shortened by $a$.



For the rest of this section, in analogy to Sela's shortening argument for surfaces, we will show that length of $X$ can also be shortened using thin components.  First, we will work on a special set of loops.



\begin{definition}
   Let $u$ be a subarc of the underlying real graph $\Gamma$ of $X$ and a point $z_{0}$ in the interior of $u$ be a basepoint. A \textbf{short loop} (with respect to $u$)  is a loop $p_1*\lambda*p_2$ based at $z_0$ where $p_1$ and $p_2$ are paths in the interior of $u$ and $\lambda$ is a path within a leaf.
\end{definition}

Let $Y_0\subset X$ be a thin component, $Y_0, Y_1,\dots$ be an infinite sequence formed by Process I and $X^i$ be the band complex containing $Y_i$ ($X=X^0$). According to \cite[Proposition 5.8]{bfstableactions}, given any non-degenerated subarc $u$ of $\Gamma$, there exists a generating set of the image of $\pi_1(Y)$ in $\pi_1(X)$ consisting of short loops with respect to $u$. In particular, we may choose $u_i\subset Y_i$ that satisfies the assumption of Corollary~\ref{thinboundary}. 

Now, let $\B_0$ be a set of short loops (with respect to $u_0$) generating the image of $\pi_1(Y_0)$ in $\pi_1(X^0)$,  and $\B^0$ (containing $\B_0$) be a set of loops generating $\pi_1(X^0)$ with the property that each $\mu \subset \B^0-\B_0$ is a loop $p_1*\widetilde{\lambda}*p_2$ based at $z_0$ where $p_1$ and $p_2$ are paths in $u_0$ and $\widetilde{\lambda}$ is a path in $Z^0$ (the closure of the complement of $Y_0$ in $X^0$). The length of $X^0$ with respect to $\B^0$ is $|X^0|_{\B^0}:=\sum_{\mu\in \B^0}m_0(\mu)$ where $m_0$ is the transverse measure on $X^0$. For every short loop $\mu \in \B_0$ of the form $p^1_{\mu}*\lambda*p^2_{\mu}$, $m_0(\mu)=m_0(p^1_{\mu})+m_0(p^2_{\mu})$. For every $\mu \in \B^0-\B_0$ of the form $p^1_{\mu}*\widetilde{\lambda}*p^2_{\mu}$, $m_0(\mu)=m_0(p^1_{\mu})+m_0(p^2_{\mu})+m_0(\widetilde{\lambda})$.

Note that weight $0$ bands do not make any contribution for the length of a band complex. We may assume that $X$ contains no weight $0$ bands. So there is a natural inclusion $\i^i: X^{i+1}\hookrightarrow X^i$ which is a homotopy equivalence. On the other hand, there exists homeomorphism $f^i: X^i\to X^{i+1}$ as in Corollary ~\ref{thinboundary}. Let $a_k: X\to X$ be a map defined as 
$$a_{k}=\i^{0}\circ\dots \circ \i^{k-1}\circ f^{k-1}\circ \dots \circ f^{0}$$

It is clear that $a_k$ is a homotopy equivalence for all $k\geq 0$. 
\vspace{1pc}

\textbf{Claim} $a_k$ is a shortening map for large $k$'s.

\begin{proof}
For each short loop $\mu \in \B_0$, $f^{k-1,0}(\mu)=f^{k-1}\circ \dots \circ f^{0}(\mu)$ is a short loop in $Y_k$ with respect to $u_k$. Since $\i^i$ preserves transverse measure for all $i$, $a_{k}(\mu)$ and $f^{k-1,0}(\mu)$ has the same measure. Therefore, 
$m_0(a_k(\mu))=m_k(f^{k-1,0}(\mu))\to 0$ as $k\to \infty$ due to measures of bands in $Y_k$ go to $0$ as $k\to \infty$. In particular, for any $\epsilon$ satisfying the following,
 
 $$\min\{m_{0}(\mu)|\mu\in \B_0\}>\epsilon>0,$$
 we may pick $k$ sufficiently large such that $m_{0}(a_k(\mu))< \epsilon$ for all $\mu \in \B_0$. 
 
 Moreover, for every $\mu=p^1_{\mu}*\widetilde{\lambda}*p^2_{\mu} \in \B^0-\B_0$, since $\i^i$ and $f^i$ are the identity on $Z^i$, for all $i$, we have that 
 $$m_0(p^1_{\mu})+m_0(p^2_{\mu})+m_0(\widetilde{\lambda})>m_k(f^{k-1,0}(p^1_{\mu}))+m_k(f^{k-1,0}(p^2_{\mu}))+m_0(\widetilde{\lambda})$$ 
 
So $m_0(\mu)>m_0(a_k(\mu))$ for $\mu\in \B^0-\B_0$ as $k\to \infty$. 
Therefore, for sufficiently large $k$, we have :
 $$ |X|_{\B^0} > \epsilon*|\B_0|+\sum_{\mu\in \B^0-\B_0}m_0(\mu)> \sum_{\mu\in \B^0}m_0(a_k(\mu))= |X|_{a_k(\B^0)}.$$ 
 Thus $X$ is shortened by $a_k$. 
 \end{proof}
 
 In fact, $m_0(\mu)$ for $\mu \in \B_0$ and  $m_0(p^1_{\mu})+m_0(p^2_{\mu})$ for $\mu \in \B^0-\B_0$ can be shortened as much as one wants to by taking $k$ sufficiently large. Moreover, every generating set of $X$ can be written in terms of $\B^0$. Therefore $X$ can be shortened with any given fixed generating set. 
\vspace{5pc}


\nocite{vgstable, vgsmall, sela2,sela3,sela4, sela5, sela6, sela7, ripsinduction}
\nocite{reyinde}
\bibliographystyle{amsalpha}
\bibliography{ref}

\end{document}